\documentclass[12pt,leqno]{article}
\usepackage{latexsym}

\usepackage{appendix}
 
 \usepackage{esvect} 
 \usepackage{latexsym}
\usepackage{amsfonts}       
\usepackage{color}
\usepackage{amsmath}
\newcounter{eqnsave}

\def\C{{{{\rm {\mbox{\small l}}} \kern -.50em {\rm C}}}}
\def\I{{{{\rm l} \kern -.10em {\rm I}}}}
\def\R{{{{\rm l} \kern -.15em {\rm R}}}}
\def\N{{{{\rm l} \kern -.15em {\rm N}}}}
\def\E{{{{\rm l} \kern -.15em {\rm E}}}}

\newcommand{\bea}{\begin{eqnarray}}
\newcommand{\beas}{\begin{eqnarray*}}

\newcommand{\ba}{\begin{array}}
\newcommand{\ea}{\end{array}}

\newtheorem{theorem}{Theorem}[section]
\newtheorem{lemma}[theorem]{Lemma}
\newtheorem{corollary}[theorem]{Corollary}

\newtheorem{definition}{Definition}[section]

\newcommand{\qed}{{}\hfill $\Box$ \\[15pt]}
\newcommand{\eea}{\end{eqnarray}}                %
\newcommand{\bas}{\begin{eqnarray*}}            %
\newcommand{\eas}{\end{eqnarray*}}              %
\begin{document}

\title{Blow up of solutions for  a Parabolic-Elliptic Chemotaxis System  with   gradient dependent chemotactic coefficient  }
\author{ J.Ignacio Tello }
\date{}
\maketitle

\begin{abstract}
We consider a Parabolic-Elliptic system of PDE's with a chemotactic term in a $N$-dimensional unit ball describing the behavior of the density of a  biological species
``$u$" and a chemical stimulus ``$v$". The system includes a  
  nonlinear chemotactic coefficient depending of ``$\nabla v$", i.e. the chemotactic term is given in the    form 
$$- div (\chi u |\nabla v|^{p-2} \nabla v),  \qquad  \mbox{ for }  \  p \in ( \frac{N}{N-1},2),   \qquad  N >2  $$
for a positive constant $\chi$   when  $v$  satisfies the poisson equation $$- \Delta v = u - \frac{1}{|\Omega|} \int_{\Omega} u_0dx.$$
We study the radially symmetric solutions under the assumption in the initial mass
$$  \frac{1}{|\Omega|} \int_{\Omega} u_0dx>6.$$
For $\chi$ large enough, we present conditions in the initial data,  such that  any regular  solution of the problem blows up at finite time.  

  \end{abstract}

 \footnotetext[2]{ Departamento de Matem\'{a}ticas Fundamentales, Facultad de Ciencias, 
Universidad Nacional de Educaci\'on a Distancia, 28040 Madrid. Spain}

\section{Introduction}

Chemotaxis is among  the most important processes in {\emph{Natural Sciences}}. It is defined as the biological phenomenon of  living organism in respond to a chemical stimulus, orientating its movement  towards the higher concentration of the chemical or away from it. 

Two main magnitudes appear in the process:
 the concentration of one or several chemical substances and the density of one or several  
biological species. 
 
 From   the pioneering works of Keller and Segel in the $70^{\prime}$s to the present,  PDE$^{\prime}$s  systems of  chemotaxis have  been studied by  a large   number of  authors.  The extensive literature in the field shows the relevance of the problem,  we refer the reader to 
Horstmann \cite{horstmann1}, \cite{horstmann2},  Bellomo et al \cite{bellomo1}, Hillen and Painter \cite{hp}    and references therein
for more details concerning  previous  results  of such systems.  
The classical system proposed in Keller and Segel  \cite{KS} and \cite{KS2} considers a linear dependence  of the chemotactic term respect to the gradient of the chemical substance.
Denoting by $u$ the living organism concentration and by $v$ the chemical stimulus, the original system  reads as follows   
$$ \begin{array}{ll}u_t - \Delta u = - div (\chi u \nabla v) + g(u,v),  &  x \in \Omega, \  t >0, 
\\ \tau v_t - \Delta v = h(u,v), 
 &  x \in \Omega, \  t >0, 
 \end{array}
$$   
for some  constants  $ \chi \in \R$,  $\tau \geq 0$,  and known  functions $g$ and $h$.  To complete the system, Neumann boundary conditions and initial data are given. 

 Nevertheless, any biological individual  presents  a natural limitation to the    velocity of movement,  and therefore, it is not expected that for large values of $\nabla v$, the behavior  of the individuals will be proportional to the case of small values of $\nabla v$. In that sense, a nonlinear term to  limit the growth of $\nabla v$  give us a large range of applications. 

Recently, motivated by different biological phenomena,  
several authors  have considered the chemotactic sensitivity coefficient ``$\chi$"  as a continuous function  of $\nabla v$,  
for instance, in Bellomo and Winkler \cite{bw1} and \cite{bw2} (see also Bellomo et al  \cite{bellomo2}), a chemotaxis  system is analyzed  for a chemotactic term in the form
 $$- div \left(  \frac{\chi  u}{[1+  |\nabla v|^2]^{\frac{1}{2}}}\nabla v\right).$$
 In  \cite{bw2}, the authors prove the existence of blow up for some initial data with nonlinear diffusive term 
 $$-  div \left(\frac{u}{\sqrt{u^2+ |\nabla u|^2}} \nabla u \right)$$
 where $v$ satisfies the elliptic problem 
 $$ - \Delta v = u- \frac{1}{|\Omega|} \int_{\Omega} u_0dx.$$
 Recently, Chiyoda, Mizukami and Yokota \cite{cmy} study the parabolic-elliptic system  for a chemotaxis term 
 $$- div \left(\frac{u^q}{\sqrt{1+ |\nabla v|^2}} \nabla v \right),$$
 where the diffusive term for ``$u$" generalizes the previous model  in \cite{bw2}, 
   $$- div \left(\frac{u^p}{\sqrt{u^2+ |\nabla u|^2}} \nabla u \right).$$
In \cite{cmy},  the authors obtain blow up of solutions for $p,q \geq 1$ by using a sub-solution method.  

 Recently, M. Winkler \cite{winkler} proves blow up of solutions for a general chemotactic term 
$$- div (\chi u f( |\nabla v|^2) \nabla v $$
for a regular function $f$ satisfying 
\begin{equation} \label{lll} f (\xi) >c (1+ |\nabla v|^2)^{-\alpha}, \mbox{ 
where }
\alpha <\frac{N-2}{2(N-1)}. \end{equation}

In Bianchi, Painter and Sherratt \cite{bps}, \cite{bps2}
 the authors consider the term
$$-  div     \left(\frac{\chi u}{(1+\omega u)}  \frac{\nabla v }{(1+ \eta |\nabla v|)}\right) ,$$
for some positive constants $\chi$,  $\omega$ and $\eta$, where a parabolic equation is  coupled to an  ODE modeling Lymphangiogenesis in wound healing in a one-dimensional spatial domain. 

We  express a general chemotactic term in the form 
$$- div  \left[u \tilde{\chi}( u, v, |\nabla v|)\nabla v\right],$$
for a prescribed continuous function  $\tilde{\chi}$.   Several authors have  studied the problem where $\tilde{\chi}$ depends only of $u$ or $v$, such examples can be found for  instance in 
Lauren{\c c}ot and Wrzosek \cite{lw}, Stinner, Tello and  Winkler \cite{stw1}, Negreanu and Tello \cite{nt1},  Stinner and Winkler \cite{sw}  and Winkler \cite{winkler1}  among others.  
In the present article we focus our attention in the case where $\tilde{\chi}$   depends only on $|\nabla v |$ in the following way  
 $$- div  \left[u \tilde{\chi}( u, v, |\nabla v|)\nabla v\right]= - \chi div  \left[u  |\nabla v|^{p-2}\nabla v\right],$$
for   some positive constant $\chi $ and $ p \in (N/(N-1),2)$ for $N>2$. 
Notice that  the chemotactic term presents a singularity when $|\nabla v|=0$.  Assumption  (\ref{lll}) in  \cite{winkler}    is equivalent  to $p \in (N/(N-1), 2)$ in this article (see assumption (\ref{p})).

The previous nonlinearity has been already studied in Negranu and Tello \cite{nt8} and in 
Wang and Li \cite{wangli}. In  \cite{nt8}, the authors consider the  system in a bounded domain $\Omega \subset \R^N$,

\begin{equation} \label{system}  \left\{  	\begin{array}{ll}  
 u_t-\Delta u= -  div  (\chi u|\nabla v|^{p-2}\nabla v), \ &   x\in\Omega,\quad t>0,
 \\     
  -\Delta v  = u-M, & x\in\Omega,\quad t>0, 
\end{array} \right. 
\end{equation} 
with homogeneous Neumann boundary conditions and   non-negative  initial data  satisfying  
\begin{equation} \label{datoinicial1} \frac{1}{|\Omega|}\displaystyle\int_{\Omega} u_0(x)dx=M. \end{equation} 
Under assumptions 
$$ \left\{  \begin{array}{ll}  p \in (1, \infty),   &   \mbox{ if }    N=1, \\
  p\in \left(1, \frac{N}{N-1}\right),  &  \mbox{ if }    N\geq 2,
\end{array}  \right. 
$$
the authors obtain uniform bounds in $L^{\infty}(\Omega)$ for any $t>0$. Similar result is obtained if $v$ satisfies 
$$- \Delta v+v=u, \quad  x\in\Omega. $$ 
The steady states 
of  the one-dimensional case are  also considered in \cite{nt8}, where   infinitely many non-constant solutions appear  for  $p\in (1,2)$ for any   $\chi$ positive and a prescribed  positive  mass.

The parabolic-parabolic equation is  considered in \cite{wangli}  for $v$ satisfying 
$$v_t- \Delta v= -uv, \quad t>0, \quad x\in \Omega 
$$
where  $\Omega $ is a $N$-dimensional bounded domain  for $N \geq 2$  with Neumann boundary conditions and bounded initial data   $u_0$ and $v_0$.  
The authors obtain global existence of weak solutions for initial data $v_0$ satisfying 
$$v_0 \leq \frac{1}{ 4 k_v} , \quad   k_v:= \sup_{s \geq 0} \left\{ \frac{s }{(1+s^2)\ln(1+s)}  \right\} $$
when the exponent $p$ satisfies 
$$N  < \frac{8-2(p-1)}{p-1}, \quad i.e. \quad  p < \frac{N+10}{N+2}.$$
Notice that the non-linear term ``$uv$'' in the equation is not equivalent to the linear term $u-v$ in \cite{nt8}.

In this article we study a mathematical prototype of chemotaxis with flux limitation   in   the $N$-dimensional open  unit ball   $B_N$ defined 
$$ B_N : = \{ x\in \R^N,  \; \;   |x|<1\},$$
  that  we denote by $B$ if not explicitly stated otherwise.
We denote by $\vv{n}$ the  outward pointing normal vector on the boundary $\partial  B $. 
The equation for $v$ is restricted to the elliptic case, for  simplicity, we assume that $v$ satisfies the Poisson equation and  the system studied is the following
\begin{align}{} 
\label{1.1} 
 &u_t-\Delta u = -  div  (\chi u|\nabla v|^{p-2}\nabla v),  & x\in B,\quad t>0, &
 \\[2mm]  \label{1.2}
 & -\Delta v  = u-M, &  x\in B,\quad t>0,  &
  \\[2mm]  \label{1.3}  &
 \displaystyle    \frac{\partial  u}{\partial \vv{n} }=\displaystyle\frac{\partial v}{\partial \vv{n} }=0, &  x\in \partial. B,\quad t>0,  &
\\[2mm]   & u(0,x)=u_0(x),  &   x\in B,  \label{1.4} &
     \end{align}      %
     where $p$ satisfies 
     \begin{equation} \label{p}
 p \in \left( \frac{N}{N-1},2\right), \mbox{ for  $ N >2$}
    \end{equation}
    and $M$, defined in (\ref{datoinicial1}),  fullfils
     \begin{equation} \label{M}
M >6,
    \end{equation}
    i.e.
   $$
\int_{B} u_0 dx >6 |B| =6 \frac{\pi^{\frac{N}{2}  }}{\Gamma( \frac{N}{2} +1)} ,
$$
    and for the three  dimensional  case,  the previous inequality   reads
$$
\int_{B} u_0 dx > 8 \pi .
$$
Notice that the previous assumption is also  assume  in  Semba \cite{semba}, where the author proves    blow up for $p=2$.

    We define the function 
    $$  \phi(0, \rho )=  \left\{
 \begin{array}{ll}   \displaystyle 
  \frac{\rho^{\gamma} }{\rho^{\gamma}+1 }  ,   & 0<\rho \leq \frac{1}{2},  
\\ [4mm] \displaystyle 
  \frac{2^{1- \gamma} }{ 2^{- \gamma} +1} \left( 1-\rho+ (1+ \gamma)\frac{ (\rho_2- \rho)_+(\rho- \frac{1}{2})}{(\rho- \frac{1}{2})} \right),   
&  \frac{1}{2} <\rho < 1, 
\end{array}
 \right.
$$
for $\gamma>1 $ satisfying 
    \begin{equation} \label{gamma} \gamma<  \min\{  1+ \frac{2-p}{N(p-1)} , \frac{ \chi    N^{-p} }{2}  -1 ,   1 + \frac{M-6}{4} ,   \frac{N+1}{N}, \gamma^* \},
\end{equation}
for $\gamma^*>1$ with the following property: 
$$\gamma^* > 1  \mbox{ such that   }  q(\gamma) >\frac{1}{5} \mbox{ for any }  \gamma\in (1, \gamma^*),$$  where   
$$q(\gamma)=  \frac{(\gamma+2) }{4(\gamma+1)^2} \left[   \frac{ 3 \gamma}{2} +1 \right].
$$
    We study the blow up of solutions   under the following assumptions  in the initial data $u_0$
     \begin{equation} \label{Hdato1}
     \mbox{ \emph{$u_0$ is a radial function, } }
     \end{equation}
\begin{equation}
\label{id23} 
u_0 \in C^{1, \alpha}  (B) , \qquad  \frac{\partial u_0 }{ \partial \vv{n} }=0, \quad x\in \partial B, 
\end{equation}
     \begin{equation} \label{Hdato2} 
      \frac{N\pi^{\frac{N}{2}}}{\Gamma(\frac{N}{2} +1)}\int_{ |x| \leq \rho^{1/N}} (u_0(x) -M) dx \geq \phi(0,\rho).
     \end{equation}
     $$ 
  \chi >  N^{p} \max\{4, 
\frac{    3 \cdot  2^{6}   }{    M -6 } \left(\frac{4}{3}\right)^{\frac{2}{N}-p} \} 
  $$ 
  the previous assumption guaranties 
     \begin{equation} \label{chi} 
      \chi_N:= \frac{\chi}{N^{p}}>   \max \left\{4, 
\frac{    \rho_2^{2- \frac{2}{N}}  2^{2}  (1+ \gamma) }{(\rho_2- \frac{1}{2})   (1- \rho_2)   (M -6) }  \left(\frac{4}{3}\right)^{2-p}  \right\} 
     \end{equation}
     in view of $\gamma<2$ and 
$$ \rho_{2}= \frac{2+ \gamma}{2(1+ \gamma)} >\frac{5}{8}.$$
  The main result of the article is enclosed in the following theorem. 
  \begin{theorem} 
  \label{t1} 
  Let $B$ be  the $N$-dimensional open  unit ball in $\R^N,$     then,  under assumption 
  (\ref{p})-(\ref{chi}), there exists a positive number $T_{bu}<\infty$ such that,  
   there exists at least   a  solution $u$  to  problem (\ref{1.1})-(\ref{1.4})  such that the function $\int_{|x|<\rho^{\frac{1}{N}}} (u(t,x)-M)dx$  exists  in $(0,T_{bu})$  
and 
  $$\lim_{t \rightarrow T_{bu} }  \| u \|_{L^{\infty}(B)}   = \infty, $$ for some $ T_{bu} \leq T_{max}:= \frac{1}{\epsilon}$ with  $\epsilon$ defined 
    as a function of $p$, $N$, $M$ and   $\chi$ in (\ref{epsilon}). 
  \end{theorem}
  Notice that, as a consequence of the previous theorem, the classical solutions of the problem, for initial data satisfying (\ref{Hdato1})-(\ref{Hdato2}) do not exists (in a classical sense)   for $t \geq  T_{bu}$.

  The article is organized as follows:   In section \ref{s2}, we  introduce  the mass accumulation function $U$ and we deduce  the  equation satisfied by $U$.   In Section \ref{s2.1}, the local existence of solutions and continuous  regularity is given for the mass accumulation function $U$.   
    Section \ref{s3} is devoted to the construction of a subsolution (denoted by $\phi$) and its properties. 
 In Section \ref{s5},  the comparison result is given for the  equation obtained in Section \ref{s2}. Uniqueness of solutions is obtained following the steps of the  comparison results given in the same section. Finally,  in the last section,  the end of the proof of the theorem is presented.   
 
 The key of the proof of the results is the sub-solution $\phi$, such function have been  constructed by   modifying the sub-solution presented in J\"ager and Luckhaus \cite{jl}
 where the authors prove finite time blow up for the minimal Keller-Segel system, i.e. for $p=2$.  
     
    %
  %
  %
  %
  %
  %
  %
  %
  
  %
  %
  
  %
  %
  %
  %
  %
  %
  %
  %
  %
  %
  %
  %
  %
  %
  %
  %
  %
  %
  %
  
  %
  %
  %
  
  %
  %
  \section{Equation of the mass accumulation  for radial symmetric solutions} \label{s2}
\setcounter{equation}{0}
 Let $u$ be the solution to  (\ref{1.1})-(\ref{1.4}) 
 and  
%
   $$\omega_{N-1}:=|S^{N-1}| = \frac{ N \pi^{\frac{N}{2}}}{\Gamma( \frac{N}{2}+1)} ,$$
    the (N-1)-dimensional volume of the sphere $S^{N-1}$ (the surface of the N-dimensional ball)   
for  the well known function $\Gamma$,  already defined by Euler in 1729  and  given by 
$$\Gamma(  z):=\int_{0}^{\infty} t^{z-1} e^{-z} dt, \qquad  {\cal{R}}e(z)>0.$$
Let  $\tilde{u}(t,r)$  be  defined by $\tilde{u}(t,|x|):=u(t,x)$ for a  radially symmetric function $u$. For simplicity we drop the tilde and  introduce the following change of unknowns
\begin{equation} \label{U}
U(t, \rho):= \int_{|x|< \rho^{1/N}} (u(t,x) -M) dx =  \omega_{N-1}\int_{0}^{ \rho^{1/N}} (u(t,r) -M) r^{N-1}dr ,
\end{equation} for $M$ defined in 
(\ref{datoinicial1}).
Notice that, thanks to Leibniz$^{\prime}$s rule we obtain 
$$ \begin{array}{lll}
U_{\rho} & =& \displaystyle  \frac{\partial }{\partial  \rho}\int_{|x|<\rho^{1/N}} (u(t,x) -M) dx
 \\ [4mm] & =& \displaystyle
 \frac{\partial }{\partial \rho}  \omega_{N-1} \int_{0}^{\rho^{1/N}} (u(t,r) -M) r^{N-1}dr 
\\ [4mm]  & =& \displaystyle  \frac{\omega_{N-1}}{N} [u(t,\rho^{1/N}) -M] \rho^{\frac{N-1}{N}}  \rho^{\frac{1}{N} -1} 
\\ [4mm] 
 & =& \displaystyle  \frac{\omega_{N-1}}{N} [u(t,\rho^{1/N}) -M] ,  \end{array}  $$ 
 therefore 
\begin{equation} \label{uss} U_{\rho \rho}=  \frac{\omega_{N-1}}{N}  \frac{\partial }{\partial \rho} u(t,\rho^{1/N}) \end{equation} 
and 
\begin{equation} \label{u}   u(t,\rho^{1/N})  =  \frac{N}{\omega_{N-1}} U_\rho +M . \end{equation} 
We have that 
 $$\begin{array}{lll} \displaystyle -\int_{|x|<\rho^{1/N} } \Delta u dx & = & \displaystyle -  \omega_{N-1} \int_{0}^{\rho^{1/N} } \left[ r^{1-N} \frac{\partial }{\partial r} \left(r^{N-1} \frac{\partial u}{\partial r} \right) \right] r^{N-1} dr  
 \\ [4mm]  & = & \displaystyle -  \omega_{N-1} \int_{0}^{\rho^{1/N} } \frac{\partial }{\partial r} \left(r^{N-1} \frac{\partial u}{\partial r} \right) dr $$
 \\ [4mm]  & = & \displaystyle -   \omega_{N-1} \rho^{\frac{N-1}{N}} \frac{\partial u}{\partial  \rho^{1/N}} $$
 \\ [4mm]  & = & \displaystyle -   \omega_{N-1} N \rho^{\frac{2N-2}{N} } \frac{\partial u}{\partial  \rho}. 
 \end{array}  $$
Thanks to (\ref{uss}) we have 
\begin{equation}
\label{deltau}
-\int_{|x|<\rho^{1/N} } \Delta u dx = - N^2 \rho^{\frac{2N-2}{N} } U_{\rho \rho} .
\end{equation}
The term 
$$\begin{array}{lll}  \displaystyle  \int_{|x|< \rho^{\frac{1}{N}}} div (   u|\nabla v|^{p-2}\nabla v) dx  & = & \displaystyle   \omega_{N-1}  
  \int_{0}^{\rho^{\frac{1}{N}}}  r^{1-N} \frac{\partial }{\partial r} \left(  r^{N-1}  u \left| \frac{\partial v}{\partial r}  \right|^{p-2}
\frac{\partial v}{\partial r} \right) r^{N-1} dr 
\\ [4mm] & = & \displaystyle   \omega_{N-1}   \int_{0}^{\rho^{\frac{1}{N}}}  \frac{\partial }{\partial r} \left(  r^{N-1}  u \left| \frac{\partial v}{\partial r}  \right|^{p-2}
\frac{\partial v}{\partial r} \right) dr 
\\ [4mm] 
& =  &  \displaystyle \omega_{N-1}    \left(  \rho^{\frac{N-1}{N}}  u \left| \frac{\partial v}{\partial  \rho^{\frac{1}{N}}}  \right|^{p-2}
\frac{\partial v}{\partial \rho^{\frac{1}{N}}} \right) 
\\ [4mm] 
& =  &  \displaystyle  \omega_{N-1}   N^{p-1}   \left(  \rho^{p \frac{N-1}{N}}  u(t,\rho^{1/N})  \left| \frac{\partial v}{\partial  \rho }  \right|^{p-2}
\frac{\partial v}{\partial \rho } \right) .
\end{array}
$$
As before we have that 
\begin{equation}
\label{deltav}
-\int_{|x|<\rho^{1/N} } \Delta v dx =-   \omega_{N-1} N \rho^{\frac{2N-2}{N} } \frac{\partial v}{\partial  \rho}  = U(\rho^{1/N},t).
\end{equation}
i.e. 
\begin{equation}
\label{v2}
\frac{\partial v}{\partial  \rho}  = - \frac{ \rho^{\frac{2-2N}{N}}}{N \omega_{N-1}}U(\rho^{1/N},t).
\end{equation}
Then
$$\begin{array}{l}  \displaystyle  -\chi \int_{|x|<\rho^{1/N}} div (   u|\nabla v|^{p-2}\nabla v) dx 
= - \omega_{N-1} \chi  N^{p-1}   \left(  \rho^{p \frac{N-1}{N}}  u(t,\rho^{1/N})  \left| \frac{\partial v}{\partial  \rho }  \right|^{p-2}
\frac{\partial v}{\partial \rho } \right)
\\ [4mm] \displaystyle 
= \omega_{N-1}^{2-p}  \chi    \rho^{p \frac{N-1}{N}  +(p-1)\frac{2-2N}{N} }   \left(  \frac{N}{\omega_{N-1}} U_\rho(t,\rho^{1/N})+M  \right)   \left| U(\rho^{1/N},t)  \right|^{p-2}
U(\rho^{1/N},t) 
\\ [4mm] \displaystyle 
= \omega_{N-1}^{2-p}  \chi    \rho^{(2-p) \frac{N-1}{N}   }   \left(  \frac{N}{\omega_{N-1}} U_\rho(t,\rho ^{1/N})+M  \right)   \left| U(t, \rho^{1/N})  \right|^{p-2}
U(t, \rho^{1/N}). 
\end{array}
$$
After integration in (\ref{1.1}), thanks to (\ref{uss})-(\ref{v2}) and the last equation,   we get 
\begin{equation}
\label{U4} \begin{array}{l}
U_t  -  N^2 \rho^{\frac{2N-2}{N} } U_{\rho \rho }= \\ [4mm]    \    \    \     \hspace{2cm}      
 \omega_{N-1}^{2-p}  \chi    \rho^{(2-p) \frac{N-1}{N}   }   \left(  \frac{N}{\omega_{N-1}} U_\rho(t,\rho^{1/N})+M  \right)   \left| U(t, \rho^{1/N})  \right|^{p-2}
U(t, \rho^{1/N}) \end{array}
\end{equation}
with the boundary condition 
$$U(t,0)=U(t,1)=0$$
and the initial data
 $$U(0,\rho)= \int_{|x|<\rho^{1/N}} (u_0-M)dx.$$
We re-escale the problem in the following way 
$$\tilde{t}=  N^2 {t}, \qquad   \tilde{U } =   \frac{N}{\omega_{N-1}}   {U}  $$ to get
 $$
\tilde{U}_{\tilde{t}}  -  \rho^{\frac{2N-2}{N} } \tilde{U}_{\rho \rho }=  
  \chi N^{-p }    \rho^{(2-p) \frac{N-1}{N}   }      \left(   \tilde{U}_{\rho} +M  \right)   \left| \tilde{U} \right|^{p-2}
\tilde{U}.  $$
For simplicity, we drop the tilde
and introduce the constant $\chi_{_N}$, already defined in (\ref{chi})  
$$ \chi_{_N}=    \chi N^{-p} 
$$
to get 
\begin{equation}
\label{eqU}
 {U}_{{t}}  -  \rho^{\frac{2N-2}{N} } {U}_{\rho \rho }=  
  \chi _{_N}    \rho^{(2-p) \frac{N-1}{N}   }   \left(  {U}_{\rho} +M  \right)   \left|  {U} \right|^{p-2}
 {U}.
\end{equation}
The problem is completed with the boundary conditions 
\begin{equation} 
\label{boundaryconditions} 
U(t,0)=U(t,1)=0
\end{equation}
and the initial data 
\begin{equation} 
\label{iniU}
U(0,\rho)= \int_{|x|<\rho^{1/N}} (u_0-M)dx.
\end{equation}

  \section{A priori estimates and local existence of weak solutions}  
  \label{s2.1} 
    \setcounter{equation}{0}
In this section we obtain some a priori estimates of the solution and prove the local existence of weak solutions for equation (\ref{eqU}).  First,   we introduce the  new  variables  $s$ and $W$ defined  by 
  $$ \rho= s^{N}, \qquad  W(t,s)= s^{-N}U(t, s^N) $$
  then,  
   $$ \begin{array}{lll}   \displaystyle \frac{\partial W }{\partial s}  & = & \displaystyle  -N s^{-N-1}U(t,s^N)+    
   N s^{-1} \frac{ \partial  U}{\partial \rho },    \\ [4mm]  \displaystyle
         \frac{\partial^2 W }{\partial s^2}  & = & \displaystyle  N(N+1) s^{-N-2}U(t,s^N)  -N(N+1) s^{-2}U_{\rho} +
   N^2s^{N-2} \frac{ \partial^2  U}{\partial \rho^2 } , \end{array}  $$  
   i.e. 
    $$  \begin{array}{rll} \displaystyle 
    \frac{ \partial  U}{\partial \rho } & = & \displaystyle  \frac{s}{N}  \frac{\partial W }{\partial s} + s^{-N}U(t,s^N) =  \frac{s}{N}  \frac{\partial W }{\partial s} + W, 
     \\ [4mm] 
     U_{\rho \rho}  & = & \displaystyle  \frac{s^{2-N}}{N^2} \frac{\partial^2 W }{\partial s^2} -\frac{N+1}{N} s^{-2N}U(t,s^N)  + \frac{N+1}{N} s^{-N}U_{\rho}
 \\ [4mm] 
 & = & \displaystyle   \frac{s^{2-N}}{N^2} \frac{\partial^2 W }{\partial s^2} -\frac{N+1}{N} s^{-N}W  +\frac{N+1}{N} s^{-N}[ \frac{s}{N}  \frac{\partial W }{\partial s} + W   ]
\\ [4mm]  & =& \displaystyle \frac{s^{2-N}}{N^2} \frac{\partial^2 W }{\partial s^2}  +
\frac{N+1}{N^2} s^{-N+1}\frac{\partial W }{\partial s} ,     
\\ [4mm]  \displaystyle \rho^{1- \frac{2}{N}}U_{\rho \rho} & = &\displaystyle \frac{1}{N^2}
 \left[ \frac{\partial^2 W }{\partial s^2}  +\frac{N+1}{s}  \frac{\partial W }{\partial s} \right].  
 \end{array}
 $$
We replace in (\ref{eqU}) and multiply by $s^{-N}$  to obtain   
  \begin{equation}
\label{eqU22}
 {W}_{{t}}  -  N^{-2} \left[ {W}_{ ss  } + \frac{N+1}{s} W_s \right] =   
  \chi _{_N}    s^{p-2   }   \left( \frac{s}{N}   W_{s} +W+M  \right)   \left|  W \right|^{p-2}
 {W}
\end{equation}
with the corresponding Dirichlet boundary conditions and initial data. 
Notice that as far as $u$ is bounded,   we have that 
$$ |U| \leq c  \rho, \quad \mbox{ and } \quad  |W| \leq c .$$  
The proof of the existence of solutions is based on the Hardy inequality 
$$ \int_0^1 \rho^{ - \delta   } u^2d \rho \leq c \int_0^1 \rho^{ 2- \delta   } |u_{\rho}|^2d \rho.$$
For readers convenience,       we introduce the details  of the proof in the following lemma. 
    \begin{lemma} \label{lemma4.2}  Let $I=(0,1)$ and $\delta \in (0, 1) \cup (1, \infty) $  then, 
    for any function $u \in H^1_{\rho^{2- \delta}}(I) $ such that $u(1)=0$ and $\lim_{\rho \rightarrow 0}  \rho^{1- \delta}u^2=0$,  we have 
     \begin{equation}  \label{eq} \int_I \rho^{ - \delta   } u^2d \rho \leq  \frac{1}{\epsilon_0( |1- \delta| - \epsilon_0)} \int_I \rho^{ 2- \delta   } |u_{\rho}|^2d \rho \end{equation}
    for any $\epsilon_0>0$ such that $$   \epsilon_0 < |1- \delta|.    $$    
   \end{lemma}
   \begin{proof}
      We consider first the case $\delta <1$ and 
   take $\epsilon_0>0 $ such that 
   $$\epsilon_0 < 1- \delta .$$ Then, 
   $$[\rho^{\epsilon_0}u ]_{\rho} = \epsilon_0 \rho^{ \epsilon_0-1} u + \rho^{\epsilon_0}u_{\rho}, 
   $$
   we take squares in the  previous equation  and it results 
   $$ 0\leq | [\rho^{\epsilon_0}u ]_{\rho}|^2 \leq \epsilon_0^2 \rho^{ 2\epsilon_0-2} u^2 + 
   \rho^{2 \epsilon_0}|u_{\rho}|^2+ 2\epsilon_0 \rho^{ 2\epsilon_0-1} u u_{\rho} .
   $$
   We now multiply by $\rho^{-2 \epsilon_0 + 2-\delta}$  and integrate over $I$, 
   $$0 \leq   \epsilon_0^2  \int_I \rho^{ - \delta } u^2 d\rho + 
  \int_I  \rho^{2 - \delta }|u_{\rho}|^2 d \rho+ 2\epsilon_0 \int_{I}  \rho^{ 1- \delta } u u_{\rho} d\rho.$$
  Since 
 $$ \begin{array}{lll}  \displaystyle 2\epsilon_0 \int_{I}  \rho^{ 1- \delta } u u_{\rho} d\rho  & =  & \displaystyle  \epsilon_0 \int_{I}  \rho^{ 1- \delta } (u^2)_{\rho} d\rho
 \\ 
    [4mm] 
    & =& \displaystyle  - \epsilon_0 (1- \delta) \int_{I}  \rho^{ - \delta } u^2d \rho
    \end{array}$$
   then, 
    $$0 \leq   \epsilon_0 (\epsilon_0 -1 + \delta  )  \int_I \rho^{ - \delta } u^2 d\rho + 
  \int_I  \rho^{2 - \delta}|u_{\rho}|^2 d \rho.$$
We divide  the previous inequality by 
$\epsilon_0 (\epsilon_0 -1 + \delta)$ and the proof of 
  (\ref{eq}) ends   for  $ \delta<1$. 
  \newline 
  To prove the case $\delta >1$ we consider
  $$\epsilon_0 <  \delta-1 $$ and the equation
   $$[\rho^{-\epsilon_0}u ]_{\rho} = -\epsilon_0 \rho^{ -\epsilon_0-1} u + \rho^{-\epsilon_0}u_{\rho}.
   $$
   We proceed as before to get 
    $$ 0 \leq \epsilon_0^2 \rho^{ -2\epsilon_0-2} u^2 + 
   \rho^{-2 \epsilon_0}|u_{\rho}|^2- 2\epsilon_0 \rho^{ -2\epsilon_0-1} u u_{\rho} .
   $$
  We multiply by 
  $\rho^{2 \epsilon_0 + 2-\delta}$  and integrate over $I$ to obtain, after integration by parts 
  $$ 0\leq  \epsilon_0^2\int_{I} \rho^{ -\delta } u^2 d \rho + 
  \int_I \rho^{2 - \delta }|u_{\rho}|^2d \rho+ \epsilon_0(1-\delta ) \int_I \rho^{ -\delta }  u^2 d \rho.
   $$
  Since $\delta >1$ and $\epsilon_0<\delta -1$, the proof ends after dividing by $\epsilon_0( \delta -1-\epsilon_0 )$.
\qed
   \end{proof}
  \begin{corollary} \label{cor4.3}
   Under assumptions of Lemma \ref{lemma4.2} we have 
$$ \int_I \rho^{ - \delta   } u^2d \rho \leq  \frac{4}{| 1- \delta|^2} \int_I \rho^{ 2- \delta   } u_{\rho}^2d \rho.
$$ 
 \end{corollary}
 \begin{proof} The proof of Corollary \ref{cor4.3} is an immediate consequence of Lemma \ref{lemma4.2} for   $\epsilon_0= \frac{|1- \delta| }{2}$.
 \qed
 \end{proof}  
   \begin{definition} 
  \label{def} 
  Let  $I=(0,1)$,  $I_T=(0,T) \times I$, then,  for any  initial data   $W_0 \in H^2_{s^{N+1}} (I) \cap H^1_{0, s^{N+1}}(I)$,   
   a {\em weak solution} of (\ref{eqU22}) in $  I_T $ is  a function
\begin{equation} \label{existencia} 
	W\in L^{\infty} ((0,T): H^1_{0,s^{N+1}} (I)) \cap H^1((0,T): L^2_{s^{N+1}} (I)) \cap L^2((0,T):
H^2_{s^{N+1}} (I) ) 
\end{equation}
and 
$$
	W\in   C^0([0,T): L^2_{s^{N+1}} (I))
$$
  such that  $W:  I_T \rightarrow \R$
  satisfies
$$ \begin{array}{l} \displaystyle -  \int_{I_T} \zeta_t W s^{N+1} ds dt
	+ \int_I\zeta(T) W(T) s^{N+1} ds  +N^{-2}  \int_{I_T}   \zeta_sW_s   s^{N+1} dsdt= \\
\\  \displaystyle
	 \int_I \zeta(0) W_0 s^{N+1} ds +  \chi_{_N} \int_{I_T}   s^{p+N-1  }   \zeta  ( \frac{s}{N}  W_s +W+M)|W|^{p-2 }W  ds dt 
	 \end{array}
$$
  for all  $\zeta \in C^{1}([0,T]: C^2_c(I))$. 
    \end{definition}
   \begin{lemma} \label{lemma2.2} 
Let $W$ the solution to (\ref{eqU22}), then,  there exists $T_1>0$ such that 
\begin{equation} \label{e1}  \left. \int_I W^2 s ds \right|_{T} +  \int_{I_T}   \left| W_s \right|^2  s ds dt  \leq C(T)<\infty, \quad \mbox{ for any $T<T_{1}$} . \end{equation}
Moreover, we have that there exists $T_2>0$ such that 
\begin{equation}
\label{e2}  \int_{I_T}  W_t^2 s^{N+1} ds dt+   \int_I  \left| W_s \right|^2  s^{N+1} ds\leq C(T) <\infty,  
\end{equation}
and for any $\epsilon \in (0, \frac{1}{pN})$ we have  
\begin{equation} \label{e3} \int_0^T  \|s^{\epsilon} W\|^2_{L^{\infty}(I)} dt \leq C(T)   ,  \end{equation}
and 
\begin{equation} \label{e4}   \int_{I_T}  \left| W_{ss}+ \frac{N+1}{s} W_s\right|^2  s^{N+1}ds dt   \leq C(T)   ,  \end{equation}
for any $T<T_2$. 
  \end{lemma} 
    \begin{proof} 
  We multiply (\ref{eqU22}) by $sW$ and integrate by parts to obtain 
  $$ \frac{1}{2} \frac{d  }{dt} \int_I W^2 s ds+ N^{-2}  \int_I \left| W_s \right|^2  sds $$ 
  $$= \frac{\chi_{_N}}{N} \int_I s^{p}|W|^{p} W_{s} ds
  + \chi_{_N}  \int_I s^{p-1} |W|^{p}Wds+ \chi_{_N}M \int_I s^{p-1} |W|^{p}ds .$$
  Since 
  $$ \int_I s^{p}|W|^{p} W_{s} ds = \frac{1}{p+1} \int_I s^{p} ( |W|^{p} W) _{s} ds
 = - \frac{p}{p+1} \int_I s^{p-1} |W|^{p} W ds
  $$
  and 
  $$\int_I s^{p-1} |W|^{p}ds \leq \int_I s^{p-1} |W|^{p+1}ds +1$$ 
  we have 
\begin{equation} \label{eqW} \frac{1}{2} \frac{d  }{dt} \int_I W^2 s ds+ N^{-2}  \int_I \left|W_s  \right|^2  sds  \leq  c  \int_I s^{p-1} |W|^{p+1} ds+  \chi_{_N}M.  \end{equation}
Let $\epsilon$ be a positive number such that 
$$\epsilon \leq \frac{1}{pN} <\frac{p-1}{p},$$ 
then  $$s^{\epsilon}  W = \int_0^s [\tau^{\epsilon}  W]_{\tau}d \tau, $$ 
therefore
     $$  \begin{array}{lll} \displaystyle |s^{\epsilon}  W | & \leq & \displaystyle   \int_I s^{\epsilon}  |W_{s}|d s+  \epsilon \int_I s^{\epsilon-1} |W|d s \\ [4mm]  &  \leq & \displaystyle \left[\int_I s^{2 \epsilon-1} ds \int_I s  |W_{s}|^2d s\right]^{\frac{1}{2}}
     +   \left[  \int_I s^{ \epsilon-1} ds \int_I s^{\epsilon-1} |W|^2d s \right]^{\frac{1}{2}}
\end{array} $$ 
which implies, in view of lemma \ref{lemma4.2} 
\begin{equation} \label{infinito}   |s^{\epsilon}  W |\leq c\left[\int_I s  |W_{s}|^2d s\right]^{\frac{1}{2}}
\end{equation} 
and the term 
\begin{equation} \begin{array}{lll} \displaystyle
  \int_I s^{p-1} |W|^{p+1} ds  & \leq &  \displaystyle  \| s^{\epsilon}W\|_{L^{\infty}(I)}^p
   \int_I s^{p-1- p\epsilon} |W| ds   \\ [4mm] & \leq &  \displaystyle   \| s^{\epsilon}W\|_{L^{\infty}(I)}^p
  \left[ \int_I s^{2(p-1- p\epsilon)-1} ds\int_I s |W|^2 ds\right]^{\frac{1}{2}}
  \end{array} 
  \label{eq3.9}
\end{equation}
as a consequence of the election of $\epsilon$,  the inequality  $2(p-1- p\epsilon)>0$ is satisfied and   we have $$ \int_I s^{2(p-1)- p\epsilon-1} ds\leq c <\infty$$
and   (\ref{eqW})  becomes, thanks to (\ref{infinito}) and (\ref{eq3.9})   
$$ \frac{1}{2} \frac{d  }{dt} \int_I W^2 s ds+ N^{-2}  \int_I \left| W_s \right|^2  sds 
\leq   c  \left[  \int_I \left| W_s \right|^2  sds \right]^{\frac{p}{2}} \left[  \int_I s |W|^{2} d s \right]^{\frac{1}{2}} 
 + \chi_{_N}M.
$$
 Thanks to Young$^{\prime}$s Inequality 
 $$
 \displaystyle \frac{1}{2} \frac{d  }{dt} \int_I W^2 s ds+ N^{-2}  \int_I \left| W_s \right|^2  sds 
  \leq  
  \frac{N^{-2}}{2}   \int_I \left| W_s \right|^2  sds   + c \left[  \int_I  |W|^{2} s d s \right]^{\frac{1}{2-p}} 
 + \chi_{_N}M. 
$$
which implies, in view of $\frac{1}{2-p} > \frac{p+1}{2}$ that 
$$
\frac{d  }{dt} \frac{1}{2}  \int_I W^2 s ds+ \frac{N^{-2}}{2}   \int_I \left| W_s \right|^2  sds 
   \leq     c \left[  \int_I  |W|^{2} s d s \right]^{  \frac{1}{2-p} }
  +c.
$$
 After integration, and thanks to Gronwall$^{\prime}$s Lemma,  we get that there exists $T_1>0$ such that 
 $$  \frac{1}{2} \int_I W^2 sds + \frac{N^{-2}}{2}  \int_{I_T} \left| W_s \right|^2  sds   \leq c(T), \quad \mbox{ for any $T<T_{1}$} $$
 and we prove (\ref{e1}). 
 \newline 
 To obtain (\ref{e2}) we multiply (\ref{eqU22}) by $s^{N+1} W_t$ and integrate by parts to obtain
   $$\begin{array}{lll} 
 \displaystyle I_W & := & \displaystyle  \int_I W_t^2 s^{N+1}ds+ \frac{N^{-2}}{2}  \frac{d  }{dt}  \int_I \left| W_s \right|^2  s^{N+1} ds  \\ [4mm] \displaystyle 
    & =  &  \displaystyle \frac{ \chi_{_N}}{ N }\int_I s^{p+N }|W|^{p-2} W W_s W_{t} ds
   + \chi_{_N}\int_I s^{p-1+N} |W|^{p} W_t ds\\ [4mm] \displaystyle 
    &   &  \displaystyle + \chi_{_N}M\int_I s^{p-1+N} |W|^{p-2}W W_t ds,
  \end{array} $$
  thanks to Young inequality,  the previous integrals are  bounded as follows
 $$\begin{array}{lll} 
 \displaystyle  \frac{ \chi_{_N}}{ N }\int_I  s^{p+N }|W|^{p-2} W W_s W_{t} ds & \leq &  \displaystyle \frac{1}{4}  \int_I \left| W_t \right|^2  s^{N+1}ds \\[4mm] & & \displaystyle +\frac{ \chi_{_N}^2}{N^2} \| s^{\frac{2p+N-1}{2p-2}} W\|_{L^{\infty}(I)}^{2p-2}   \int_I \left| W_s \right|^2  s^{N+1}ds 
  \end{array} $$  
  $$\begin{array}{lll} 
 \displaystyle  \chi_{_N}\int_I s^{p+N-1} |W|^{p} W_t ds
  & \leq   & \displaystyle   \frac{1}{4}  \int_I \left| W_t \right|^2  s^{N+1}ds+
    \chi_{_N}^2\int_I s^{2p+N-3} |W|^{2p}  ds\\  [4mm] & \leq &  \displaystyle 
     \frac{1}{4}  \int_I \left| W_t \right|^2  s^{N+1}ds+
    \chi_{_N}^2\| s^{\frac{2p+N-4}{2p-2}} W \|^{2p-2}_{L^{\infty}(I)} \int_I s W^2  ds
    \end{array} $$  
  $$\begin{array}{lll} 
 \displaystyle   \chi_{_N}M\int_I s^{p+N-1} |W|^{p-1} |W_t| ds
  & \leq   & \displaystyle \frac{1}{4}  \int_I \left| W_t \right|^2  s^{N+1}ds+  \chi_{_N}^2 M^2 \int_I s^{2p+N-3} |W|^{2p-2}  ds\\  [4mm] & \leq &  \displaystyle 
     \frac{1}{4}  \int_I \left| W_t \right|^2  s^{N+1}ds+
    \chi_{_N}^2 M^2 \left(  \int_I W^2s  ds + 1 \right)
  , \end{array} $$
  where the last inequality is a consequence of $s^{2p+N-4} \leq 1$
Notice that $ {N} \geq 3$ and therefore 
  $$2p+N-4>0.$$ 
  Now, we apply   (\ref{infinito})  and (\ref{e1}) and it results   
 $$I_W \leq \frac{3}{4}  \int_I \left| W_t \right|^2  s^{N+1}ds+ 
 d(t) \left[  \int_I \left| W_s  \right|^2  s^{N+1}ds+1 \right] + c
 $$ 
 where $$d(t):= c \left[ \int_I \left|\frac{\partial  W }{\partial s} \right|^2  s ds+1\right]^{p-1 }  \in L^{\frac{1}{p-1}}(0,T_1) .$$
 We apply Gronwall$^{\prime}$s lemma to the previous inequality to prove
  that  for any  $T < T_1 $ there exists $c(T)$ such that 
 $$ 
 \int_{I_T}   W_t^2 s^{N+1}dsdt+   \int_I \left| W_s\right|^2  s^{N+1}ds\leq c(T) <\infty,   $$  
 which  proves   (\ref{e2}). 
 \newline 
(\ref{e3}) is a consequence of (\ref{infinito}) and (\ref{e2}). 
\newline 
Finally, to obtain (\ref{e4}) we first multiply by $s^{N+1}\left[ {W}_{ ss  } + \frac{N+1}{s} W_s \right]$ the following equation     $$
  -  N^{-2} \left[ {W}_{ ss  } + \frac{N+1}{s} W_s \right] =  -  {W}_{{t}}+
  \chi _{_N}    s^{p-2   }   \left( \frac{s}{N}   W_{s} +W+M  \right)   \left|  W \right|^{p-2}
 W,
$$and integrate over $I$.
 Now, we proceed as before,   and apply Young$^{\prime}$s inequality to the right-hand-side terms, after integration over $(0,T)$ and thanks to (\ref{e2}) and (\ref{e3}) the proof of  the lemma ends as a consequence of the following inequality 
 $$  \begin{array}{lrl} \displaystyle \frac{1}{2} \int_{I_T} \left|W_{ss} \right|^2  s^{N+1}ds  dt  \leq &  \displaystyle   \int_{I_T} \left| W_{ss}
+\frac{N+1}{s} W_s\right|^2  s^{N+1}ds dt   & \\ [4mm] & \displaystyle + 2(N+1)^2\int_{I_T}  \left| W_s\right|^2  s^{N-1}ds  dt & \leq C(T)  . \end{array}$$
   \qed \end{proof}
      \begin{lemma}
Let $N \geq 3$, and    $W_0 \in H^2_{s^{N+1}} (I) \cap H^1_{0, s^{N+1}}(I)$,
 then, there exists 
  $T_{{bu}}>0$ and  at least a  weak solution $W$   to  (\ref{eqU22}) in $(0, T_{bu})$.  
  \end{lemma} 
    \begin{proof} We first consider $T>0$,   such that $ T< \min\{T_1,T_2\} $   and $k> 0$ large satisfying 
    $$ k >1+  C(T)  \quad t\leq T,  $$  for $C(T)$ defined in Lemma \ref{lemma2.2}.
     We define the subset $$Q:=\{  W \in L^2((0,T): L^2_{s^{N+1}}(I)),  \quad  \int_I \left| W_s \right|^2s^{N+1}  ds < k   \}.$$ 
         For a given $W^{n-1} \in Q$, we consider the   problem \begin{equation} \label{iterat}
 \left\{ \begin{array}{l} {W}_{{t}}^n  -  N^{-2} \left[ {W}_{ss  }^n + \frac{N+1}{s} W_s^n \right]  \\ [4mm]   \displaystyle =   
  \chi _{_N}    s^{p-2   }   \left( \frac{s}{N}  W_{s}^n +W^n+M  \right)   \left|  W^{n-1} \right|^{p-2}
  W^{n-1} 
\\[4mm]
W^n(t,0)=W^n(t,1)=0 \\[4mm]
W^n(0,s)=W_0(s).
\end{array} \right.
\end{equation} 
In  analogous fashion to Definition \ref{def},  we define the notion of weak solution to  (\ref{iterat}) 
i.e. $W^n$ is a weak solution to (\ref{iterat}), if for any $W^{n-1} \in Q$, $W^n$  satisfies  
\begin{equation} \label{solW}  \begin{array}{l} \displaystyle
- \int_{I_T}  \zeta_t W^n s^{N+1} ds dt
	+  \int_I \zeta(T) W^n(T) s^{N+1} ds +N^{-2}  \int_{I_T}    \zeta_sW_s^n  s^{N+1} ds dt= \\
\\  \displaystyle \int_I \zeta(0) W^n_0 s^{N+1} ds+
	 \chi_{_N} \int_{I_T}   s^{p+N-1  }   \zeta  ( \frac{s}{N}  W^n_s +W^n+M)| W^{n-1} |^{p-2 } W^{n-1}  ds  dt
	 \end{array}
\end{equation} 
  for all   $\zeta \in C^{1}([0,T]: C^2_c(I))$. 
We construct a functional 
 $J(W^{n-1})=W^n$ where $W^n$ is the solution to  (\ref{iterat}).
  We obtain,  in the same way  as in   lemma \ref{lemma2.2},   the following estimates 
 \begin{equation}  \label{f2}  \int_{I_T}  | W_t^n|^2 s^{N+1}dsdt+   \int_I \left|  W^n_s  \right|^2  s^{N+1}ds dt \leq c(T) <\infty,  
\end{equation}
\begin{equation} \label{f4}   \int_{I_T}   \left| W^n_{ss}+ \frac{N-1}{s}   W^n_s \right|^2  s^{N+1}ds dt   \leq C(T) ,  \end{equation}
\begin{equation} \label{f5}    \|  s^{N+1}  W^n  \|_{L^{\infty}(I)}    \leq C(T) .  \end{equation}
(\ref{f5}) implies that
$$(s^{N+1} W^n)_k= s^{N+1}W^n,   \quad (s^{\frac{N+1}{p-2}} U)_k= s^{\frac{N+1}{2-p}}W^n. $$
%
Let $H^2_{rad} (B_{N+2})$ be  the Sobolev functional space of radially symmetric functions   in $L^2(B_{N+2})$ 
  defined over  the   $N+2$-dimensional unit  ball $B_{N+2}$, with derivatives  in $L^2(B_{N+2})$   up to order two.   Since $$H^2_{rad} (B_{N+2})  \equiv H^2_{s^{N+1}}(I) $$
  (see \cite{defiguereido}   Theorem 2.3) and $H^2(B_{N+2})  \hookrightarrow   H^1 (B_{N+2})$ is a compact embedding we have that   $ H^2_{s^{N+1}}(I) \hookrightarrow   H^1_{s^{N+1}}(I)$ and $ H^1_{s^{N+1}}(I) \hookrightarrow   L^2_{s^{N+1}}(I)$ are also compact.
Now,    Aubin-Lions Theorem   and Schauder fixed point Theorem,  provides the existence  of a fixed point  $W^*$,   
which is a  weak solution of (\ref{solW}) for $T$ small enough. 
 It is possible to extend the solution  as far as $W$ satisfies (\ref{existencia}), i.e. there exists a $T_{bu}$ such that there exists a weak solution to  (\ref{eqU22}) in $(0, T_{bu}) \times I$.
     \qed \end{proof}
     Now, we introduce the notion of weak solution to (\ref{eqU})-(\ref{iniU}). 
        \begin{definition} 
  \label{def2} 
  Let  $I=(0,1)$ and  $U_0 \in H^2_{s^{2- \frac{2}{N}}} (I) \cap H^1_{0 }(I)  \cap L^2_{0, s^{-2}}(I)$.  
  Then,   a {\em weak solution} of (\ref{eqU})-(\ref{iniU}) in $(0,T)  \times I $ is  a function
$$
	U\in  L^2((0,T):H^1_{0} (I))  \cap H^1((0,T):L^2_{\rho^{\frac{2}{N}-2} }(I)) \cap L^2((0,T):L^{2p}_{\rho^{p\frac{2}{N}-2} }(I)) 
$$
  such that  $U:[0,T]\times I \rightarrow \R$
  satisfies
$$ \begin{array}{l} \displaystyle
	- \int_{I_T} \zeta_t U \rho^{\frac{2}{N}-2 } d\rho dt + \int_{I} \zeta (T) U(T) \rho^{\frac{2}{N}-2 } d\rho 
	 +  \int_{I_T}   \zeta_{\rho}   U_{ \rho}  d \rho dt= \\ [4mm] \displaystyle
	  \int_{I} \zeta (0) U_0 \rho^{\frac{2}{N}-2 } d\rho 
	 \chi_{_N} \int_{I_T}  \rho^{ \frac{-p(N-1)}{N} }   \zeta [  U_{\rho}  +M]|U|^{p-2 }U  d\rho dt 
	 \end{array}
$$  for all     $\zeta \in C^{1}([0,T]: C^2_c(I))$. 
    \end{definition} 

        \begin{lemma}
Let $N \geq 3$, and    $U_0  \in H^2_{s^{2- \frac{2}{N}}} (I) \cap H^1_{0 }(I)  \cap L^2_{0, s^{-2}}(I)$,
 and $U$ a weak solution to    (\ref{eqU})-(\ref{iniU}),  then, 
  for any  $T\in(0,T_{bu}) $ we have 
  $$\begin{array}{l} \displaystyle \int_I U^2  \rho^{\frac{2}{N}-3}d \rho   \leq C(T), 
   \\ [4mm] 
  \displaystyle 
\int_{I_T}       |U_t|^2 \rho^{\frac{2}{N}-2}  d\rho dt +
      \int_I \left|   U_{\rho}\right|^2  d\rho + \int_I    U^2  
   \rho^{-2}d\rho   
 \leq C(T) ,
   \\ [4mm] 
  \displaystyle 
    \int_{I_T}  \rho^{2- \frac{2}{N}}  \left|   U_{\rho \rho }\right|^2  d\rho dt   \leq C(T)
  \\ [4mm] 
  \displaystyle 
 | U(t, \rho) | \leq c \rho^{\frac{1}{2}} C(T),
\end{array}
$$ moreover 
$U \in C((0,T):C_0^{0,\frac{1}{2}}(I)),$
  \end{lemma} 
    \begin{proof} 
Since $ W^2= (s^NW)^2 s^{-2N} $ and $sds= \frac{s^{2-N}}{N} [ N s^{N-1} ds]  $ we have   $$  \int_I W^2 sd s 
= \frac{1}{N}  \int_I (W s^{N}) ^2 s^{2-3N}   Ns^{N-1} d s =  \frac{1}{N} \int_I U^2  \rho^{\frac{2}{N}-3}d \rho   
$$ which implies 
  \begin{equation}  \label{g1}   \int_I U^2  \rho^{\frac{2}{N}-3}d \rho\leq c(T) <\infty. 
\end{equation}
We also have that,  $s ds= N^{-1}\rho^{\frac{2}{N}-1}d\rho$,   and therefore   
$$ \begin{array}{lll} \displaystyle    \int_I \left| W_s \right|^2  sds 
&  =  & \displaystyle \frac{1}{N}   \int_I \left|  \frac{N}{\rho^{\frac{1}{N}}} \left( U_{\rho} -\frac{U}{\rho} \right)
  \right|^2 \rho^{\frac{2}{N}-1}  d \rho \\[4mm] 
&  =  & \displaystyle  N    \int_I \left| U_{\rho}  -\frac{U}{\rho}
  \right|^2 \rho^{-1} d \rho 
\end{array}
  $$
 and 
it results 
  \begin{equation} \label{g2} \int_{I_T}  \left|   U_{\rho} - \frac{U}{\rho} \right|^2 \rho^{-1} d\rho  dt \leq C(T)   .  \end{equation}
  In the same  fashion,   we obtain, thanks to (\ref{e2})  and (\ref{g1})   
\begin{equation} \label{g3} 
\int_{I_T}       |U_t|^2 \rho^{\frac{2}{N}-2}  d\rho dt +
      \int_I \left|   U_{\rho}\right|^2  d\rho + \int_I    U^2  
   \rho^{-2}d\rho   \leq C(T)    \end{equation}
 which implies, in view of the embedding $H^1_0(I) \hookrightarrow L^{\infty}(I)$, 
   \begin{equation} \label{g5}   | U |   \leq C(T).   \end{equation}
As before, from (\ref{e4}) we deduce  
 \begin{equation} \label{g4} 
     \int_{I_T}  \rho^{2- \frac{2}{N}}  \left|   U_{\rho \rho }\right|^2  d\rho dt   \leq C(T).    \end{equation}
 Notice that thanks  to Cauchy-Schwarz inequality  
$$ U(t, \rho) = \int_{0}^{\rho} U_{r} dr \leq \rho^{\frac{1}{2}} \left[\int_{0}^{\rho} |U_{r}|^{2}dr\right]^{\frac{1}{2}} \leq  \rho^{\frac{1}{2}} \left[\int_I |U_{r}|^{2}dr\right]^{\frac{1}{2}} $$
which implies 
  \begin{equation} \label{g7} | U(t, \rho) | \leq c \rho^{\frac{1}{2}} C(T),    
  \end{equation} 
   for any $t \leq T<T_{{bu}}.$ From (\ref{g3}) we have that 
   $$U \in L^{\infty}(0,T:H_0^1(I)) \cap H^{1}(0,T:L^2_{\frac{2}{N}-2}(I))  .  $$
   Thanks to the compact  embedding
   $$H^1_0(I) \hookrightarrow C^{0, \frac{1}{2}}_0(I)$$ 
   where $C^{0, \frac{1}{2}}_0(I)$ denotes the H\"older continuous functions in $I$ with  zero boundary values 
   and Aubin-Lions Lemma, we have that 
   \begin{equation} \label{regul}
   U \in C((0,T):C_0^{0,\frac{1}{2}}(I)), \quad \mbox{ for any } T < T_{bu} ,
   \end{equation}
   and the proof ends.
   \qed    
   \end{proof}
  %
  %
  %
  %
  %
  %
  %
  
        \begin{lemma}
Let $N \geq 3$, and    $U_0  \in H^2_{s^{2- \frac{2}{N}}} (I) \cap H^1_{0 }(I)  \cap L^2_{0, s^{-2}}(I)$,
 then, there exists 
  $T_{bu}>0$ and  at least a  weak solution $U$   to   (\ref{eqU})-(\ref{iniU}),  such that $$ \limsup_{t \rightarrow T_{bu}} \|U\|_{L^{\infty}}+ t= \infty.$$   
    \end{lemma} 
    \begin{proof} 
           Notice that,  since $\zeta, U \in L^2_{\rho^{\frac{2}{N}-2}} (I)$  and thanks to Cauchy-Swartz and Young inequalities the term   
    $$ \begin{array}{lll} \displaystyle      \int_I \rho^{ \frac{-p(N-1)}{N} }   \zeta |U|^{p-2 }U  d\rho 
 &    \leq  & \displaystyle \frac{1}{2}  \int_I  \rho^{\frac{2}{N}-2} \zeta^2  d\rho + \frac{1}{2} 
  \int_I \left[  \rho^{ \frac{2}{N} -2 }     |U|^2 \right]^{p-1}  d\rho
  \\ [4mm] 
  & \leq &  \displaystyle \frac{1}{2}  \int_I  \rho^{\frac{2}{N}-2} \zeta^2  d\rho + \frac{1}{2} 
 \left[  \int_I   \rho^{ \frac{2}{N} -2 }     |U|^2  d\rho \right]^{p-1} 
   \end{array}
     $$ is bounded.  Now we see the boundedness of the term 
     $$
     \int_{I_T}   \rho^{ \frac{-p(N-1)}{N} }   \zeta   U_{\rho} |U|^{p-2 }U  d\rho dt
     $$ in the weak formulation (\ref{def2}).  As before,  we first  apply Cauchy-Schwartz and Young inequalities to obtain 
       $$   \int_I \rho^{ \frac{-p(N-1)}{N} }   \zeta  U_{\rho}  |U|^{p-2 }U  d\rho 
    \leq \frac{1}{2}\int_I    |U_{\rho}|^2 d\rho+ \frac{1}{2}   \int_I \rho^{-2p \frac{(N-1)}{N} }   \zeta^{2}  |U|^{2 (p-1)  } d\rho  $$
    $$\leq \frac{1}{2}\int_I     |U_{\rho}|^2 d\rho+ 
     \frac{c}{2}\left[ \int_I [\rho^{- \frac{N-1}{N}  }   \zeta]^{2p}    d\rho\right]^{\frac{1}{p}}  
     \left[ \int_I [ \rho^{-\frac{N-1}{N} }   U]^{2p } d\rho\right]^{\frac{p-1}{p}}  .$$
     Thanks to  (\ref{g7}) we may replace the term $U^{2p} $ by $\rho^{p-1}U^2$ in the last integral of the previous inequality, and it results  
   $$  \int_I [ \rho^{-\frac{N-1}{N} }   U ]^{2p } d\rho \leq  C(T) \int_I  \rho^{ -p-1 +\frac{ 2p}{N} }   U^{2 } d\rho. $$
   Now, (\ref{g1}) implies the boundedness of the last term  for $p\in(1,2)$ and we get 
    $$  \int_I [ \rho^{-\frac{N-1}{N} }   U]^{2p } d\rho \leq C(T), \quad \mbox{ for $t\leq T.$}$$
    In view of    
   $$  U(t,\rho)=  \rho W(t, \rho^{\frac{1}{N}}) $$
   we replace into Definition \ref{def} to obtain that $U$ is a weak solution of  (\ref{eqU})-(\ref{iniU}). 
   \qed
    \end{proof}
  \section{Constructing a subsolution}  
  \label{s3} 
    \setcounter{equation}{0}

%
We introduce the operator $\mathcal L$ defined as follows 
\begin{equation} \label{L} \mathcal{L} (\phi):= 
  \phi_{t}  -   \rho^{\frac{2N-2}{N} } \phi_{ \rho \rho}- 
   \chi_{_N}    \rho^{(2-p) \frac{N-1}{N}   }   \left(  \phi_{ \rho} +M  \right)   \left| \phi  \right|^{p-2}
\phi. 
\end{equation} 
Let $\rho_1$ and $\rho_2$  the following  positive numbers  
 \begin{equation}  \label{rho1}   \rho_1=\frac{1}{2}   
\end{equation}
and 
 \begin{equation}  \rho_2=\frac{2+ \gamma}{2(1+ \gamma)}  . \label{rho2} 
\end{equation}
Notice that in view of $\gamma>1$, (\ref{rho2}) guarantees 
$$\rho_2 \in \left(\frac{1}{2}, \frac{3}{4} \right).$$
We consider the function $q$, already defined in the introduction,  
$$q(\gamma):= \frac{(\gamma+2) }{4(\gamma+1)^2} \left[   \frac{ 3 \gamma}{2} +1 \right]  
$$
and the positive real numbers    
 $\gamma^*$,  $\gamma_0$ and $\gamma$   as follows 
 $$\gamma^* > 1,   \mbox{ such that   }  q(s) >\frac{1}{3} \mbox{ for any }  s  \in (1, \gamma^*),$$ 
  
\begin{equation} \label{gamma0} \gamma_0:= \min\{  1+ \frac{2-p}{N(p-1)} ,    \chi_{_N} -1 ,   1+\frac{M-6}{4} ,   \frac{N+1}{N}, \gamma^* \}
\end{equation} and 
$$\gamma \in \left( 1,  \gamma_0 \right) .$$
Then, we construct   the  following  function 
\begin{equation} \label{phi1} 
 \phi_1(\rho ,t)=   
   \frac{      \rho^{ \gamma}   }{     \rho^{\gamma}+ a(t)    }   ,  \quad  0<\rho<\rho_1 
\end{equation} 
where 
$$a(t):=   (1-  \epsilon  t)^{\frac{1}{1- \theta}} , \quad a^{\prime}(t)= - \frac{\epsilon }{1-\theta} a^{\theta} (t),  $$
for 
$$\theta:=  \frac{3-p}{2}>  2-p , \quad \mbox{ for }  p\in (1,2)   $$
and $\epsilon$ satisfying
\begin{equation} \label{epsilon}
\begin{array}{l} \epsilon\leq \min\left\{  \frac{\chi_{_N} \gamma (1- \theta) }{2^p  },  \ \frac{ \chi_{_N} \gamma  (p-1) }{2},  \ 
 (p-1) \rho_1^{\gamma  \frac{(p-1)}{2}}  \left[  \chi_{_N}  (\frac{M}{2} -3)\left(\frac{4}{3} \right)^{p-2}   -  \frac{2(1+ \gamma)  \rho_2^{2- \frac{2}{N}}}{(\rho_2- \rho_1)   (1- \rho_2)}    \right]
     \right\}.
     \end{array}
\end{equation}

\begin{lemma} \label{l3.1} 
Let $\phi_1$ be defined in (\ref{phi1}) and the differential operator $\mathcal{L}$   in (\ref{L}), then, under assumptions   (\ref{p})-(\ref{chi}), 
 we have that 
$$ \mathcal{L}(\phi_1) \leq 0,   \quad  0<\rho<\rho_1.$$
Where  $\mathcal{L}(\phi_1) $ is understood in the sense of distributions. 
\end{lemma}
\begin{proof}
We first  compute the following  derivatives of $\phi_1$ 
$$
\begin{array}{rll}  \displaystyle 
\phi_{1 t}  & = &   \displaystyle 
 - \frac{ \rho^{ \gamma}  a^{\prime} (t)  }{[ \rho^{\gamma}  +a(t)]^2}   ,    
\\ [4mm] \displaystyle
\phi_{1 \rho} & = &   \displaystyle    \gamma  \frac{ a(t)   \rho^{\gamma-1}          }{[\rho^{\gamma}+ a(t)]^2}   ,   
 \\[4mm] \displaystyle
 \phi_{1 \rho \rho } & = &   \displaystyle 
  \gamma  \frac{(\gamma-1) a^2 (t)   \rho^{\gamma-2}  -      a(t) ( \gamma +1) \rho^{2\gamma-2}        }{[\rho^{\gamma}+ a(t)]^3}     ,  
\\ [4mm]  \displaystyle
  -  \rho^{\frac{2N-2}{N} }\phi_{1 \rho \rho } & = &   \displaystyle    -  \gamma  \frac{(\gamma-1) a^2 (t)   \rho^{\gamma-\frac{2}{N}}  -     a(t) (\gamma +1)  \rho^{2\gamma-\frac{2}{N}}        }{[\rho^{\gamma}+ a(t)]^3}     , 
\\ [4mm] \displaystyle
 \rho^{(2-p) \frac{N-1}{N}   }  M  |\phi_1 |^{p-2} \phi_1 & = &   \displaystyle  M \frac{ \rho^{ (p-1)( \gamma -1 + \frac{1}{N} )   +1-  \frac{1}{N}  }   }{  [\rho^{\gamma   }+ a(t)]^{p-1}}   ,   
 \\[4mm] \displaystyle
\rho^{(2-p) \frac{N-1}{N}   }   \phi_{1 \rho}    |\phi_1 |^{p-2} \phi_1 & = &   \displaystyle   \frac{ a(t) \gamma   \rho^{ (p-1)( \gamma -1 + \frac{1}{N} )   +\gamma -  \frac{1}{N}  }   }{  [\rho^{\gamma   }+ a(t)]^{p+1}}    .
\end{array}
$$
Then, 
$$ \begin{array}{lll} \mathcal{L} (\phi_1) &= & \displaystyle  
  \phi_{1t}  -   \rho^{\frac{2N-2}{N} } \phi_{1 \rho \rho}- 
   \chi_{_N}    \rho^{(2-p) \frac{N-1}{N}   }   \left(   \phi_{1 \rho} +M  \right)   \left| \phi_1  \right|^{p-2}
\phi_1 
\\ [4mm] &= & \displaystyle  
 - \frac{ \rho^{\gamma}  a^{\prime} (t) }{[ \rho^{\gamma}  +a(t)]^2}  
  -  \gamma  \frac{(\gamma-1) a^2 (t)   \rho^{\gamma-\frac{2}{N}}  -     a(t) (\gamma +1  ) \rho^{2\gamma-\frac{2}{N}}        }{[\rho^{\gamma}+ a(t)]^3}  
\\ [4mm] & & \displaystyle  
 -         \chi_{_N}      \frac{ a(t) \gamma   \rho^{ (p-1)( \gamma -1 + \frac{1}{N} )   +\gamma -  \frac{1}{N}  }   }{  [\rho^{\gamma   }+ a(t)]^{p+1}} 
 -    \chi_{_N} M
  \frac{ \rho^{ (p-1)( \gamma -1 + \frac{1}{N} )   +1-  \frac{1}{N}  }   }{  [\rho^{\gamma   }+ a(t)]^{p-1}} 
\\ [4mm] &\leq  & \displaystyle  
- \frac{ \rho^{\gamma}   a^{\prime} (t)  }{[ \rho^{\gamma}  +a(t)]^2}  
+ \gamma(\gamma+1)  \frac{  a(t)   \rho^{2 \gamma-\frac{2}{N}}       }{[\rho^{\gamma}+ a(t)]^3}   -     \chi_{_N}      \frac{ a(t) \gamma   \rho^{ (p-1)( \gamma -1 + \frac{1}{N} )   +\gamma -  \frac{1}{N}  }   }{  [\rho^{\gamma   }+ a(t)]^{p+1}}
 \\ [4mm] && \displaystyle  
   -    \chi_{_N} M
  \frac{ \rho^{ (p-1)( \gamma -1 + \frac{1}{N} )   +1-  \frac{1}{N}  }   }{  [\rho^{\gamma   }+ a(t)]^{p-1}}    
.
 \end{array} $$ 
 We now   cancel the    positive   terms in the right hand side. We first split the  first negative term in the following way   
 $$ -       \chi_{_N}      \frac{ a(t) \gamma   \rho^{ (p-1)( \gamma -1 + \frac{1}{N} )   +\gamma -  \frac{1}{N}  }   }{  [\rho^{\gamma   }+ a(t)]^{p+1}} 
 $$ $$= -     \frac{   \chi_{_N} }{2}     \frac{ a(t) \gamma   \rho^{ (p-1)( \gamma -1 + \frac{1}{N} )   +\gamma -  \frac{1}{N}  }   }{  [\rho^{\gamma   }+ a(t)]^{p+1}} 
-     \frac{ \chi_{_N} }{2}     \frac{ a(t) \gamma   \rho^{ (p-1)( \gamma -1 + \frac{1}{N} )   +\gamma -  \frac{1}{N}  }   }{  [\rho^{\gamma   }+ a(t)]^{p+1}}. $$
 To cancel the term   $ - \frac{ \rho^{\gamma}  a^{\prime}  }{[ \rho^{\gamma}  +a(t)]^2}  $ we proceed as follows: 
 $$
  - \frac{ \rho^{\gamma}  a^{\prime}(t)  }{[ \rho^{\gamma}  +a(t)]^2}  
  =
  \frac{  \epsilon \rho^{\gamma}  a^{ \theta} (t) }{ (1- \theta) [ \rho^{\gamma}  +a(t)]^2}  
 $$
 and consider  three different  cases.  
 \begin{itemize} 
 \item[Case 1.]  $a(t) \geq \rho^{\gamma}$.
  $$ \frac{  \epsilon \rho^{\gamma}  a^{ \theta } (t) }{ (1- \theta ) [ \rho^{\gamma}  +a(t)]^2}  
\leq   \frac{  \epsilon \rho^{\gamma}   [ \rho^{\gamma}  +a(t)]^{\theta }  }{ (1- \theta)[ \rho^{\gamma}  +a(t)]^{2 }}  \leq 
 \frac{ \epsilon  \rho^{\gamma}   }{(1- \theta) [ \rho^{\gamma}  +a(t)]^{2- \theta }} 
 $$
 and 
 $$ \begin{array}{lcl}  \displaystyle 
 -     \frac{ \chi_{_N} }{2}     \frac{ a(t) \gamma   \rho^{ (p-1)( \gamma -1 + \frac{1}{N} )   +\gamma -  \frac{1}{N}  }   }{  [\rho^{\gamma   }+ a(t)]^{p+1}} &  \leq  & \displaystyle
 -     \frac{ \chi_{_N} }{2}     \frac{ a(t) \gamma   \rho^{ (p-1)( \gamma -1 + \frac{1}{N} )   +\gamma -  \frac{1}{N}  }   }{ 2 a (t) [\rho^{\gamma   }+ a(t)]^{p}}  \\ [4mm] 
 & = & \displaystyle
 -     \frac{ \chi_{_N} }{4}     \frac{  \gamma   \rho^{ (p-1)( \gamma -1 + \frac{1}{N} )   +\gamma -  \frac{1}{N}  }   }{ [\rho^{\gamma   }+ a(t)]^{p}}
 .
\end{array} $$
 In view of 
 \begin{equation} \label{abc1}  (p-1)( \gamma -1 + \frac{1}{N} )    -  \frac{1}{N}<0, \quad \mbox{ i.e. }  \gamma< 1+ \frac{2-p}{N(p-1)} \end{equation} 
 and for $$\epsilon \leq  \frac{\chi_{_N} \gamma (1- \theta) }{2^p  } $$
 provided $$2- \theta \leq p$$
 we have that 
 $$ \frac{  \epsilon \rho^{\gamma}  a^{ \theta} (t) }{ (1- \theta) [ \rho^{\gamma}  +a(t)]^2}  
\leq        \frac{ \chi_{_N} }{2}     \frac{ a(t) \gamma   \rho^{ (p-1)( \gamma -1 + \frac{1}{N} )   +\gamma -  \frac{1}{N}  }   }{  [\rho^{\gamma   }+ a(t)]^{p+1}}, 
\quad \mbox{ for  }   a(t) \geq  \rho^{\gamma}  .$$
 \item[Case 2.]  $ \rho^{\gamma\frac{p-1}{1- \theta}} \leq  a(t) \leq \rho^{\gamma}$.
 \newline Notice that 
$$ \frac{p-1}{1- \theta} \geq 1$$ in view of 
$$\theta  \geq 2-p$$
and therefore 
$ \rho^{\gamma\frac{p-1}{1- \theta}}  \leq \rho^{\gamma}$.
Then 
 $$ \begin{array}{lcl}  \displaystyle 
 -     \frac{ \chi_{_N} \gamma  }{2}     \frac{ a(t)   \rho^{ (p-1)( \gamma -1 + \frac{1}{N} )   +\gamma -  \frac{1}{N}  }   }{  [\rho^{\gamma   }+ a(t)]^{p+1}} &  \leq  & \displaystyle
 -     \frac{ \chi_{_N} \gamma  }{2}     \frac{ a(t)    \rho^{ (p-1)( \gamma -1 + \frac{1}{N} )   +\gamma -  \frac{1}{N}  }   }{ 2^{p-1}  \rho^{\gamma (p-1)}  [\rho^{\gamma   }+ a(t)]^{2}}  \\ [4mm] 
 & \leq  & \displaystyle
 -     \frac{ \chi_{_N} \gamma  }{2^{p} }     \frac{  a(t)  \rho^{ (p-1)( \gamma -1 + \frac{1}{N} )   +\gamma -  \frac{1}{N}  }   }{ a^{1- \theta}  [\rho^{\gamma   }+ a(t)]^{2}} \\ [4mm]
  & \leq  & \displaystyle
 -     \frac{ \chi_{_N} \gamma  }{2^{p} }     \frac{  \ a^{\theta} (t)  \rho^{ (p-1)( \gamma -1 + \frac{1}{N} )   +\gamma -  \frac{1}{N}  }   }{  [\rho^{\gamma   }+ a(t)]^{2}}
\end{array} $$
 for $$\epsilon < \frac{\chi_{_N}\gamma (1- \theta)  }{2^p} $$ 
 and thanks to (\ref{abc1}) we claim 
 $$ \frac{  \epsilon \rho^{\gamma}  a^{ \theta} (t) }{ (1- \theta) [ \rho^{\gamma}  +a(t)]^2}  
\leq       \frac{ \chi_{_N} }{2}     \frac{ a(t) \gamma   \rho^{ (p-1)( \gamma -1 + \frac{1}{N} )   +\gamma -  \frac{1}{N}  }   }{  [\rho^{\gamma   }+ a(t)]^{p+1}}, 
\quad \mbox{ for  }    \rho^{\gamma\frac{(p-1)}{1-\theta} } \leq  a(t) \leq  \rho^{\gamma}  .$$
   \item[Case 3.]  $a(t) <\rho^{\gamma \frac{p-1}{1- \theta}}$
   \newline 
   Notice that $a(t) <\rho^{\gamma \frac{p-1}{1- \theta}}$
 is equivalent to 
 $$ \frac{ a(t) }{\rho^{\gamma \frac{p-1}{1- \theta}}}<1,  \quad  \mbox{ i.e.  }  \frac{ a^{\theta}(t) }{\rho^{\gamma \frac{(p-1) \theta}{1- \theta}}}<1, $$
then
 $$  
  \begin{array}{lcl}  \displaystyle 
  -    \chi_{_N} M
  \frac{ \rho^{ (p-1)( \gamma -1 + \frac{1}{N} )   +1-  \frac{1}{N}  }   }{  [\rho^{\gamma   }+ a(t)]^{p-1}} 
 & \leq  & \displaystyle
 -     \chi_{_N}  M    \frac{  a^{\theta}(t)   \rho^{ (p-1)( \gamma -1 + \frac{1}{N} )   +1  -  \frac{1}{N}  }   }{    \rho^{\gamma \frac{(p-1) \theta}{1-\theta  }   } [\rho^{\gamma   }+ a(t)]^{p-1}}  \\ [6mm]
  & \leq  & \displaystyle
 -    \chi_{_N}M   \frac{  a^{\theta}(t)  \rho^{ (p-1)( \gamma -1 + \frac{1}{N} )   +1 -  \frac{1}{N}  }   }{  [\rho^{\gamma   }+ a(t)]^{(p-1)( 1+ \frac{\theta}{1- \theta} )  }}
  \\ [6mm]
  & \leq  & \displaystyle
 -     \chi_{_N} M      \frac{  a^{\theta}(t)  \rho^{ (p-1)( \gamma -1 + \frac{1}{N} )   + 1 -  \frac{1}{N}  }   }{  [\rho^{\gamma   }+ a(t)]^{\frac{p-1}{1- \theta}   }}
\end{array} $$
 since $\theta =  \frac{3-p}{2}$ we have that 
  $$   -    \chi_{_N} M
  \frac{ \rho^{ (p-1)( \gamma -1 + \frac{1}{N} )   +1-  \frac{1}{N}  }   }{  [\rho^{\gamma   }+ a(t)]^{p-1}} 
   \leq 
  -     \chi_{_N}M     \frac{  a^{\theta}(t)  \rho^{ (p-1)( \gamma -1 + \frac{1}{N} )   +\gamma -  \frac{1}{N}  }   }{  [\rho^{\gamma   }+ a(t)]^2 }
  $$ 
  then 
   $$ \frac{  \epsilon \rho^{\gamma}  a^{ \theta} (t) }{ (1- \theta) [ \rho^{\gamma}  +a(t)]^2}  
\leq          \chi_{_N} M    \frac{   \rho^{ (p-1)( \gamma -1 + \frac{1}{N} )   +1-  \frac{1}{N}  }   }{  [\rho^{\gamma   }+ a(t)]^{p-1}}, 
\quad \mbox{ for  }  a(t) <  \rho^{\gamma\frac{(p-1)}{1-\theta} }   .$$
  \end{itemize} 
  Case I, II and III implies 
    \begin{equation} \label{eq22}
     \frac{ \rho^{\gamma}  a^{\prime}  }{[ \rho^{\gamma}  +a(t)]^2}  -     \frac{\chi_{_N }}{2}     \frac{ a(t) \gamma   \rho^{ (p-1)( \gamma -1 + \frac{1}{N} )   +\gamma -  \frac{1}{N}  }   }{  [\rho^{\gamma   }+ a(t)]^{p+1}} -  \chi_{_N} M    \frac{   \rho^{ (p-1)( \gamma -1 + \frac{1}{N} )   +1-  \frac{1}{N}  }   }{  [\rho^{\gamma   }+ a(t)]^{p-1}}  \leq 0.
  \end{equation}

The second term is  expressed  in the following way:
$$  \begin{array}{lll} \displaystyle 
\gamma (\gamma +1) \frac{  a(t)   \rho^{2 \gamma-\frac{2}{N}}       }{[\rho^{\gamma}+ a(t)]^3}
& = &  \displaystyle   \gamma (\gamma +1)    \frac{  a(t)        \rho^{2\gamma-\frac{2}{N}}     }{[\rho^{\gamma}+ a(t)]^{ 1+p  }  } 
[\rho^{\gamma}+ a(t)]^{ p-2  } \\ [4mm] \displaystyle 
&  \leq  & \displaystyle   \gamma  (\gamma +1)     \frac{  a(t)        \rho^{2\gamma-\frac{2}{N}+     \gamma (p-2) } }{[\rho^{\gamma}+ a(t)]^{ 1+p  }  } 
 \\ [4mm] \displaystyle 
& =  & \displaystyle   \gamma  (\gamma +1)     \frac{  a(t)        \rho^{p \gamma-\frac{2}{N}  } }{[\rho^{\gamma}+ a(t)]^{ 1+p  }  } 
\end{array} $$
and 
$$
 -     \frac{ \chi_{_N} }{2}     \frac{ a(t) \gamma   \rho^{ (p-1)( \gamma -1 + \frac{1}{N} )   +\gamma -  \frac{1}{N}  }   }{  [\rho^{\gamma   }+ a(t)]^{p+1}} 
 = -        \frac{  a(t) \gamma (\gamma+1)   \rho^{p\gamma - \frac{2}{N} }    }{  [\rho^{\gamma   }+ a(t)]^{p+1}}   
 \frac{  \chi_{_N}    }{2( \gamma +1)}   \rho^{  1+ p( -1 + \frac{1}{N} )  } 
$$
thanks to assumption (\ref{p}) 
$$
 1+ p( -1 + \frac{1}{N} )  < 1+ \frac{N}{N-1} ( -1 + \frac{1}{N} ) =0.  
$$
Then, for $$\gamma <      \frac{ \chi_{_N} }{2}  -1$$ we have that 
$\phi_1$ satisfies  $$\mathcal{L}(\phi_1)\leq 0, \qquad  \mbox{ in $\rho <\rho_1.$}$$
 \qed 
 \end{proof} 
We now consider the function $\phi_2$ defined in $(\rho_1,1)$  as follows  
\begin{equation} \label{phi2} \phi_2(\rho ,t): = \beta(t) \left( 1- \rho 
+\kappa \frac{(\rho_2 - \rho)_+(\rho-\rho_1)}{(\rho_2-\rho_1)}\right)
\end{equation}
for  $$\rho_2:= \frac{2+ \gamma}{2(\gamma+1)} $$
where  
$$\beta(t)=  \phi_1(t, \rho_1) \left(1- \rho_1 \right)^{-1}  = 2\frac{2^{- \gamma}}{2^{- \gamma} +a(t)}\leq 2 $$
and \begin{equation} \label{kappa} \kappa= 1+ \gamma. \end{equation}
Notice that the function 
$$\left( 1- \rho 
+\kappa \frac{(\rho_2 - \rho)_+(\rho-\rho_1)}{(\rho_2-\rho_1)}\right)$$
attains its maximum at 
$$\frac{\rho_1+ \rho_2 }{2} -\frac{\rho_2- \rho_1 }{2k }= \frac{1}{2 \kappa} \left(\rho_1 ( \kappa + 1 ) +\rho_2 (  \kappa-1 )  \right) = \frac{(\gamma+2) (2\gamma+1)}{4(\gamma+1)^2}
$$
and at this point 
$$ \begin{array}{lll} \phi_2 &= &\beta(t) \left[ 1- \frac{(\gamma+2) (2\gamma+1)}{4(\gamma+1)^2} + \frac{1+ \gamma}{\frac{1}{2(\gamma+1)}}\left[
(\frac{\gamma+ 2}{2(\gamma+1 )}  - \frac{(\gamma+2) (2\gamma+1)}{4(\gamma+1)^2} )(  \frac{(\gamma+2) (2\gamma+1)}{4(\gamma+1)^2}- \frac{1}{2}   )  \right] \right].
\end{array}
$$
 After some  computations we get  
$$ \begin{array}{lll} \phi_2 & = & \beta(t) \left[ 1-\frac{(\gamma+2) (2\gamma+1)}{4(\gamma+1)^2}  + \frac{(\gamma+2)\gamma}{8(\gamma+1)^2}  
   \right]  .
\end{array}
$$
which implies 
$$ \begin{array}{lll} \phi_2 & = & \beta(t) \left[ 1-\frac{(\gamma+2) }{4(\gamma+1)^2} \left[   \frac{ 3 \gamma}{2} +1 \right]
   \right]  .
\end{array}
$$
Since $$ \lim_{\gamma \rightarrow 1} \frac{(\gamma+2) }{4(\gamma+1)^2} \left[   \frac{ 3 \gamma}{2} +1 \right] = \frac{15}{32} > \frac{1}{3}$$ there exists $\gamma^{*}>1$ 
such that 
$$ \frac{(\gamma+2) }{4(\gamma+1)^2} \left[   \frac{ 3 \gamma}{2} +1 \right] \geq  \frac{1}{3} >0$$
for any $\gamma\in [1, \gamma^*]$.
Therefore for such $\gamma$ we have that  
 $$\phi_2 < \frac{ 2 \beta(t)}{3}.   $$

\begin{lemma} \label{l3.2} 
Let $\phi_2$ be defined in (\ref{phi2}) and the differential operator $\mathcal{L}$   in (\ref{L}), then, under assumptions  (\ref{p})-(\ref{Hdato2}), for 
$\chi_{_N} $ large enough, satisfying 
$$ \chi_{_N} > \frac{    \rho_2^{2- \frac{2}{N}}  2^{5-p}  (1+ \gamma) }{(\rho_2- \frac{1}{2})   (1- \rho_2)   (M -6) }    $$ 
we have that 
$$ \mathcal{L}(\phi_2) \leq 0,   \quad  0<\rho_1<\rho<1.$$
Where  $\mathcal{L}(\phi_2) $ is understood in the sense of distributions. 
\end{lemma}
\begin{proof}
In order to obtain $\mathcal{L}(\phi_2)$   we  compute the derivative of $\phi_2$ 
$$
\begin{array}{ll} 
(\phi_{2})_{t}  &=   
 -  \frac{ \rho_1^{ \gamma}  a^{\prime}   (t) }{[ \rho_1^{\gamma}  +a(t)]^2}     \left( 1- \rho + \kappa \frac{(\rho_2 - \rho)_+(\rho-\rho_1)}{(\rho_2- \rho_1)}\right)  \left(1- \rho_1 \right)^{-1}  
 \\[4mm] \displaystyle
&  = -  \frac{    a^{\prime}   (t) }{[ \rho_1^{\gamma}  +a(t)]} \phi_2 
 \\[4mm] \displaystyle
&  =  \frac{ \epsilon a^{\theta}  (t)  }{(p-1)  (\rho_1^{\gamma} +a(t)) }   \phi_2, 
\\[4mm] \displaystyle
(\phi_{2})_{ \rho} & =  \left\{ \begin{array}{ll}   -  \beta(t) ( 1 + \kappa \frac{2 \rho -(\rho_2 + \rho_1) }{(\rho_2- \rho_1)} )  ,   &  \rho_1 <\rho< \rho_2, 
 \\  -  \beta (t)  ,  &  \rho_2 <\rho< 1,
   \end{array} \right.
   \\ [4mm] & \geq
 - (1+\kappa)  \beta (t) ,  \quad  \rho_1 <\rho< 1,
 \\[4mm] \displaystyle
 (\phi_{2})_{\rho \rho }  & = \left\{ \begin{array}{ll}   -  \frac{2 \kappa  \beta (t) }{(\rho_2-\rho_1)}    ,   &  \rho_1 <\rho< \rho_2,
 \\  0 ,  &  \rho_2 <\rho< 1, 
   \end{array} \right.
 \end{array}
$$
 $$
\begin{array}{lcl} 
  - 
   \chi_{_N}    \rho^{(2-p) \frac{N-1}{N}   }   \left(   \phi_{2 \rho} +M  \right)   \left| \phi_2  \right|^{p-1}
  & \leq  & - 
   \chi_{_N}    \rho^{(2-p) \frac{N-1}{N}   }   (M-  (1+ \kappa) \beta(t))   \left| \phi_2  \right|^{p-1}.
\end{array}
$$
Notice that $(\phi_2)_{\rho}$ presents a positive jump at $\rho=\rho_2$ for any $\kappa >0$, then, we have 
\begin{equation} \label{eq10} 
\begin{array}{ll}
\mathcal{L}( \phi_2)  \leq &     \displaystyle        \frac{\epsilon a^{ \theta}  (t) \phi_2 }{(p-1)  (\rho_1^{\gamma} +a(t)) }     \\ [4mm]  & \displaystyle +  \frac{2 \kappa  \beta (t) }{(\rho_2-\rho_1)} 
\chi_{[\rho_1,\rho_2]}  \rho^{2- \frac{2}{N}} 
  -    \chi_{_N}     \rho_2^{(2-p) \frac{N-1}{N}} (M-(1+ \kappa )\beta(t))\left| \phi_2   \right|^{p-1}.
  \end{array}
  \end{equation}
In  view of $a (t)\leq 1$ and  $\theta= (3-p)/2$  it results 
$$
  \frac{\epsilon a^{\theta}   (t)\phi_2 }{(p-1)  (\rho_1^{\gamma} +a(t)) }  \leq  \frac{\epsilon   \phi_2 }{(p-1)  (\rho_1^{\gamma} +a(t))^{\frac{p-1}{2}} } .
  $$
Since 
$$\phi_2 \geq \beta(t) (1- \rho) , $$  
the second term is bounded as follows
$$ \begin{array}{lll} \displaystyle 
  \frac{2 \kappa  \beta (t) }{(\rho_2-\rho_1)} 
\chi_{[\rho_1,\rho_2]}  \rho^{2- \frac{2}{N}} 
& =  & \displaystyle  \frac{2 (1+ \gamma)  \beta (t) }{(\rho_2-\rho_1)}\chi_{[\rho_1,\rho_2]}  \rho^{2- \frac{2}{N}} 
\\ [4mm] &
\leq  & \displaystyle   
\frac{2 (1+ \gamma)  }{(\rho_2-\rho_1)}\chi_{[\rho_1,\rho_2]}  \rho^{2- \frac{2}{N}} \frac{\phi_2}{(1-  \rho )}
  \\ [4mm] &
\leq  & \displaystyle  \frac{2 (1+ \gamma)    }{(\rho_2- \rho_1)   (1- \rho_2)}    \rho_2^{2- \frac{2}{N}}  \phi_2.  
\end{array}
$$
Then,   thanks to  the previous computations,  for  $$M >  (1+ \kappa ) \beta, $$ we get
$$   \begin{array}{l} \displaystyle 
\mathcal{L}( \phi_2)     \\ \displaystyle 
 \leq   \phi_2 \left[  \frac{\epsilon   }{(p-1)  (\rho_1^{\gamma} +a(t))^{\frac{p-1}{2} } }+  \frac{2(1+ \gamma) }{    (\rho_2- \rho_1) (1- \rho_2)}    \rho_2^{2- \frac{2}{N}}  
 -    \chi_{_N}    (M-(1+ \kappa)\beta (t)) \left| \phi_2   \right|^{ p-2} \right] .
 \end{array}
$$
 Since $p<2$   and $$\phi_2 <\frac{2}{3} \beta   \leq   \frac{2}{3(1- \rho_1)}  = \frac{4}{3},   $$ 
 it results 
\begin{equation} \label{second} \begin{array}{l} \mathcal{L}( \phi_2) \leq  \\  [4mm]   \displaystyle \phi_2  [  \frac{\epsilon   }{(p-1) \rho_1^{\gamma 
\frac{(p-1)}{2}}  }  
+ \frac{2 (1+ \gamma) }{(\rho_2- \rho_1)   (1- \rho_2)}    \rho_2^{2- \frac{2}{N}} 
 -    \chi_{_N}  (M- 2(1+ \kappa)) \left(\frac{4}{3}\right)^{p-2}   ].
 \end{array} 
  \end{equation} 
 provided 
 $$ M >   2(1+\kappa) = 2(2+ \gamma ).$$
Thanks to (\ref{gamma0})     we have 
$$\gamma \leq 1+\frac{M- 6}{4}   $$
and the term 
$$  -    \chi_{_N}  (M- 2(1+ \kappa)) \left(\frac{4}{3}\right)^{p-2}  \leq  -    \chi_{_N}  (M- 2(3+\frac{M-6}{4}))  \left(\frac{4}{3}\right)^{p-2}    = -    \chi_{_N}  (\frac{M}{2} -3)  \left(\frac{4}{3}\right)^{p-2} $$
then, for 
$$\chi_N >   \frac{2 (1+ \gamma) }{(\rho_2- \rho_1)   (1- \rho_2)}   \left(\frac{4}{3}\right)^{2-p}  \rho_2^{2- \frac{2}{N}} (\frac{M}{2} -3)^{-1}    =
\frac{    \rho_2^{2- \frac{2}{N}}     2^2(1+ \gamma) }{(\rho_2- \rho_1)   (1- \rho_2)   (M -6) }   \left(\frac{4}{3}\right)^{2-p}  
$$
 and 
 $$\epsilon <   (p-1) \rho_1^{\gamma  {(p-1)}{2}}  \left[  \chi_{_N}  (\frac{M}{2} -3) \left(\frac{4}{3}\right)^{p-2}  -  \frac{2(1+ \gamma) }{(\rho_2- \rho_1)   (1- \rho_2)}    \rho_2^{2- \frac{2}{N}} \right] 
  $$
we have 
  $$\mathcal{L}( \phi_2) \leq 0, \qquad  \rho_1 <\rho <1,$$
and the proof ends.
 \qed
 \end{proof} 

 
  %
  %
  %
  %
  
  %
  %
  %
  
  %
  \section{A comparison Lemma  and uniqueness of solutions}  
  \label{s5}
  \setcounter{equation}{0}
  In this section we prove a comparison lemma and  uniqueness of solutions under suitable assumptions in the initial data. We first present some previous   results. 

   \begin{lemma}
	\label{lemma4.1}  Let $p$ be a positive constant satisfying  $ p \in (1, 2)$ (i.e. $p-1 \in (0,1)$) and $x$ and $y$ be positive numbers such that 
	 $$x\geq 0, \quad   y>0, \quad   x< k y,$$  for some $k \in (0,1).$   Then, 
	\begin{equation*}
	y^{p-1} -x^{p-1} =(p-1) \xi^{p-2}  (y-x),
	\end{equation*}
	where   $\xi$   satisfies 
	\begin{equation} \label{xi} \xi \geq k_0  y  , \quad \mbox{ for  \  }
		 k_0 :=  \left( (1-p)(1-k) \right)^{\frac{1}{p-2} } .\end{equation}
\end{lemma}
\begin{proof}   The proof is a direct application of Mean Value Theorem, where $\xi$ satisfies
$$\xi^{ 2-p } = \frac{ (p-1)(y-x)}{ y^{p-1}-x^{p-1}} \geq \frac{(p-1)(1-k )y  }{  y^{p-1} }
= (1-p)(1-k) y^{2-p}  ,  $$
i.e.,
$$\xi    \geq \left(  (1-p)(1-k)\right)^{\frac{1}{2-p}} y $$
which ends the proof. 
\qed
  \end{proof}
      \begin{lemma}
	\label{lemma4.3}   
Let   $\epsilon_1$,  $\delta$ and $\kappa$ any strictly positive numbers,  such that    
$$ 0< \epsilon_1 \leq \kappa .$$ 
Let  $I=(0,1)$, and $u \in L^2_{\rho^{- \delta}}(I)$, then,    

		\begin{equation} \label{rho} 
		\int_I \rho^{\epsilon_1 -\delta} u^2 d\rho \leq \rho_0^{\epsilon_1} \int_{I} \rho^{- \delta } u^2 d \rho +   
		\rho_0^{\epsilon_1-\kappa} \int_{I} \rho^{\kappa- \delta } u^2 d \rho 
		 \end{equation}
		 for any $\rho_0 \in (0,1)$.
\end{lemma}
\begin{proof} 
 Since 
 $$ \begin{array}{lll} \displaystyle 
 \int_I \rho^{\epsilon_1 -\delta} u^2 d\rho  & = & \displaystyle \int_0^{\rho_0} \rho^{\epsilon_1 -\delta} u^2 d\rho
 +\int_{\rho_0}^1 \rho^{\epsilon_1 -\delta} u^2 d\rho,
 \\ [4mm]  \displaystyle  \int_0^{\rho_0} \rho^{\epsilon_1 -\delta} u^2 d\rho & \leq &  \displaystyle   \rho_0^{\epsilon_1} \int_{0}^{\rho_0} \rho^{- \delta } u^2 d \rho , 
  \\ [4mm]   \displaystyle  \int_{\rho_0}^1 \rho^{\epsilon_1 -\delta} u^2 d\rho  &  \leq   &  \displaystyle    \rho_0^{\epsilon_1- \kappa} \int_{\rho_0}^1 \rho^{\kappa - \delta } u^2 d \rho 
\end{array}  $$
 hold, we obtain (\ref{rho})
in view of 
$$  \int_0^{\rho_0} \rho^{- \delta } u^2 d \rho  \leq  \int_I \rho^{ - \delta } u^2 d \rho $$
and 
$$  \int_{\rho_0}^1 \rho^{\kappa - \delta } u^2 d \rho  \leq  \int_I \rho^{\kappa - \delta } u^2 d \rho $$
for any $\rho_0 \in (0,1)$.
 \qed 
 \end{proof}
   \begin{lemma} \label{lemma4.4}
   Let $p >1$  and $U$ a solution to (\ref{eqU})   with the initial data $U_0$ satisfying 
   $$U_0 \geq \phi(0,x).$$ 
  Then, under assumptions of Theorem   (\ref{t1}),    
   the solution $U$ satisfies, $$U (t,\rho) \geq \phi(t,\rho), \quad  \rho \in (0,  1), \quad t < T_{bu}:=  \frac{1}{\epsilon}$$
   for $\epsilon$ defined in (\ref{epsilon}). 
   \end{lemma}
   \begin{proof}
   We proceed by contradiction and assume that there exists $t_0 \in (0, T_{bu})$ and $\rho_3 \in (0,1)$  such that 
$$ U (t_0,\rho_3) < \phi(t_0,\rho_3). $$
Then, due to the continuity of $U$ we have that 
\begin{equation} \label{contradiction}   \int_{I_T} \rho^{-\delta}(\phi(t, \rho)-U(t, \rho))_+^2 d\rho dt >0. 
\end{equation}
We denote by 
  $w$ the difference between  $\phi$ and $U$, i.e.  
  $$w= \phi-U$$
  which satisfies the equation 
  $$
 w_t  -  \rho^{\frac{2N-2}{N} } {w}_{\rho \rho }  \    \leq    \
  \chi_{_N}   \rho^{(2-p) \frac{N-1}{N}   }  \frac{1}{p} 
  \left(   (\phi ^{p})_{\rho} -(U^{p})_{\rho}   \right) 
  + M
  \chi_{_N}   \rho^{(2-p) \frac{N-1}{N}  }   \left(     |\phi|^{p-2}  \phi-    |U|^{p-2}  U\right)
  $$
  in the  sense of distributions. 
  We now multiply by $\rho^{ \frac{2}{N} -\delta } w_+$ for some
\begin{equation} \label{delta} \delta \in (1,\gamma   )
\end{equation}
  then, after integration  over $  I_T:= I \times (0,T)  \mbox{ for }  I=(0,1) $ we get 
 \begin{equation} \label{eq1}  \begin{array}{ll} \displaystyle 
  \frac{1}{2} \int_{I} w_+^2 \rho^{\frac{2}{N} - \delta } d\rho \big|_T
  &  \    \leq   \  \displaystyle  \int_{I_T}   \rho^{2- \delta}    {w}_{\rho \rho } w_+ d\rho dt
 \\ [4mm] & \displaystyle 
 + \frac{ \chi_{_N} }{p} \int_{I_T}  \rho^{(2-p) \frac{N-1}{N}  +\frac{2}{N} -\delta }  
  \left( (  \phi^{p})_{\rho} - (U^{p})_{\rho}   \right) w_+ d\rho dt  + \\ [4mm] & \displaystyle 
    M
  \chi_{_N}  \int_{I_T}     \rho^{(2-p) \frac{N-1}{N}  + \frac{2}{N}- \delta }   \left(   \left|  {\phi} \right|^{p-2}  \phi-   \left|  {U} \right|^{p-2}  U\right) w_+ d\rho dt .
  \end{array}
  \end{equation} 
  For simplicity we label  the integrals in the previous equation in the following way 
  $$\begin{array}{lll}
  I_1 & := & \displaystyle  \int_{I_T}  \rho^{ 2- \delta }  {w}_{\rho \rho } w_+ d\rho dt,  \\ [4mm]
  I_2 & :=  & \displaystyle   \frac{ \chi_{_N} }{p}  \int_{I_T}   \rho^{(2-p) \frac{N-1}{N}  +\frac{2}{N} -\delta }  
  \left(  ( \phi^{p})_{\rho} - (U^{p})_{\rho}   \right) w_+ d\rho dt ,
  \\  [4mm] I_3 &  := & \displaystyle M
  \chi_{_N} \int_{I_T}  \rho^{(2-p) \frac{N-1}{N}  + \frac{2}{N}- \delta }   \left(   \left|  {\phi} \right|^{p-2}  \phi-   \left|  {U} \right|^{p-2}  U\right) w_+ d\rho dt.
  \end{array}
      $$
      Then, (\ref{eq1})  is  expressed
      as follows
      \begin{equation} \label{eq2}    
  \frac{1}{2} \int_{I} w_+^2 \rho^{\frac{2}{N} - \delta } d\rho  \big|_T =  I_1+ I_2+ I_3.  \end{equation} 
  Notice that,   since $$\lim_{\rho \rightarrow 0^+ } \phi_{\rho} \phi \rho^{2- \delta}
  =  \lim_{\rho \rightarrow 1^- } \phi_{\rho} \phi \rho^{2- \delta} =0, $$
   and 
   $$\lim_{\rho \rightarrow 0^+ }   \phi^2 \rho^{1- \delta}
  =  \lim_{\rho \rightarrow 1^- }   \phi^2 \rho^{1- \delta} =0, $$
  we have 
   $$ \begin{array}{lcl}  I_1  & =  & \displaystyle   \int_{I_T} \rho^{ 2- \delta }  {w}_{\rho \rho } w_+ d\rho  dt 
  \\ [4mm] & =  & \displaystyle
     -   \int_{I_T}   \rho^{2- \delta }  |{w_+}_{\rho }|^2d\rho   dt  -(2- \delta )   \int_{I_T}  \rho^{1- \delta }  (w_+)_{\rho } w_+d\rho  dt
       \\ [4mm] & =  & \displaystyle
   -    \int_{I_T} \rho^{2- \delta }  |{w_+}_{\rho }|^2d\rho   dt -  \frac{2- \delta}{2}   \int_{I_T}   \rho^{1- \delta}  (w_+)^2_{\rho } d\rho   dt
       \\ [4mm] & =  & \displaystyle
   -   \int_{I_T}  \rho^{2- \delta }  |{w_+}_{\rho }|^2d\rho  dt + \frac{(2- \delta)(1-\delta ) }{2}   \int_{I_T} \rho^{-\delta}  (w_+)^2  d\rho    dt
     \end{array}
       $$    
      Therefore  \begin{equation}
   \label{I1}
I_1=-   \int_{I_T}  \rho^{2- \delta }  |{w_+}_{\rho }|^2d\rho     dt +   \frac{(2- \delta)(1-\delta ) }{2}     \int_{I_T} \rho^{-\delta }  (w_+)^2  d\rho   dt .
   \end{equation}
  For simplicity in the notation, we label the previous integrals in the following way 
   \begin{equation}
   \label{I1a}
I_{1a}=-    \int_{I_T} \rho^{2- \delta }  |(w_+)_{\rho }|^2d\rho      dt    \end{equation}
 and 
     \begin{equation}
   \label{I1b}
I_{1b}=   \frac{(2- \delta)(1-\delta ) }{2}     \int_{I_T} \rho^{-\delta }  w_+^2  d\rho    dt .
   \end{equation}
      Now we  integrate by parts  in  $I_2$,
       $$ \begin{array}{lcl}  I_2  & =  & \displaystyle   \frac{\chi_{_N}}{2p}     \int_{I_T} \rho^{(2-p) \frac{N-1}{N}+ \frac{2}{N} -\delta  }  
  ( \phi^{p}_{\rho} - U^{p}_{\rho})    w_+ d\rho    dt      \\ [4mm] &  =  & \displaystyle 
-  \frac{\chi_{_N}}{p}     \int_{I_T}  \rho^{(2-p) \frac{N-1}{N}+ \frac{2}{N} -\delta  }  
  ( \phi^{p}  - U^{p} )_+   (w_+)_{\rho} d\rho    dt 
   \\ [4mm] &    & \displaystyle  
-((2-p) \frac{N-1}{N}+ \frac{2}{N} -\delta )   \frac{\chi_{_N}}{p}      \int_{I_T} \rho^{(2-p) \frac{N-1}{N}+ \frac{2}{N} -\delta -1 }  
     ( \phi^{p}  - U^{p} )_+     w_+  d\rho    dt ,
  \end{array}
  $$
  where the boundary terms    $$\rho^{(2-p) \frac{N-1}{N} + \frac{2}{N} - \delta} (\phi^p - U^p)_{+} w_+, \quad   \mbox{ at } \rho=0,1$$ 
  are equal to $0$ as a consequence of $|\phi| < c\rho^{\gamma}  $ and $\delta <\gamma$.
  We label the previous integrals in the following way 
   $$\begin{array}{lll}
  I_{2a} & = & \displaystyle
 -\frac{\chi_{_N}}{p}      \int_{I_T}  \rho^{(2-p) \frac{N-1}{N}+ \frac{2}{N} -\delta  }  
    ( \phi^{p}  - U^{p} )_+    (w_+)_{\rho} d\rho    dt  \\ [4mm]
  I_{2b} & :=  & \displaystyle   -((2-p) \frac{N-1}{N}+ \frac{2}{N} -\delta )   \frac{\chi_{_N}}{p}     \int_{I_T}  \rho^{(2-p) \frac{N-1}{N}+ \frac{2}{N} -\delta -1 }  
   ( \phi^{p}  - U^{p} )_+   w_+  d\rho    dt 
  \end{array}
      $$
to  write
  $$I_2= I_{2a} + I_{2b}.
     $$
    Notice that,  thanks to Young inequality
     $$\begin{array}{lll}
    I_{2a} & := & \displaystyle   \frac{\chi_{_N}}{p}     \int_{I_T} \rho^{(2-p) \frac{N-1}{N}+ \frac{2}{N} -\delta  }  
   ( \phi^{p}  - U^{p})_+    (w_+)_{\rho} d\rho    dt 
      \\ [4mm]
 & =  & \displaystyle   
   \frac{\chi_{_N}}{p}      \int_{I_T}  \rho^{  2-\delta  -p \frac{N-1}{N}  }   ( \phi^{p}  - U^{p} )_+ (w_+)_{\rho}
    d\rho    dt 
   \\ [4mm]
 & \leq  & \displaystyle  
 \frac{1}{2}     \int_{I_T}  \rho^{2 -\delta   }   | (w_+)_{\rho} |^2 d\rho     dt +
c 
    \int_{I_T} \rho^{ 2 -\delta -2p\frac{N-1}{N}   }   ( \phi^{p}  - U^{p} )_+^2
    d\rho    dt 
     \\ [4mm]
  & \leq   & \displaystyle   \   \frac{1}{2}     \int_{I_T}   \rho^{2 -\delta   }   | (w_+)_{\rho} |^2 d\rho     dt +
 c
      \int_{I_T} \rho^{ 2 -\delta  -2p\frac{N-1}{N}    } \phi^{2p-2} w_+^2
    d\rho    dt  .
    \end{array}$$
    Notice that the last inequality is obtained by Mean Value theorem since $U \leq \phi$.
     In view of  $$ \phi^{2p-2}   \leq 
 c \frac{\rho^{2p-2}}{a^{\frac{2(p-1)}{\gamma}}(t)}$$ we have that 
    $$
\begin{array}{lll}      \displaystyle 
      \int_{I_T}  \rho^{ 2 -\delta -2p\frac{N-1}{N}   } \phi^{2p-2} w_+^2
    d\rho      dt  & \leq  &  \displaystyle c    \int_{I_T}  a^{- \frac{2(p-1)}{\gamma}}(t)    \rho^{ 2 -\delta -2p\frac{N-1}{N}  +2p-2  } w_+^2
    d\rho     dt 
    \\ [4mm] & = &  \displaystyle 
    c    \int_{I_T}  a^{- \frac{2(p-1)}{\gamma}} (t)   \rho^{  -\delta + \frac{2p}{N}   } w_+^2
    d\rho    dt .
    \end{array}
    $$
 Then, we have 
 \begin{equation} \label{eq3}
 I_{2a} \leq  \frac{1}{2}    \int_{I_T} \rho^{2 -\delta   }    | (w_+)_{\rho} |^2 d\rho     dt +c
     \int_{I_T} a^{- \frac{2(p-1)}{\gamma}}(t)     \rho^{  -\delta + \frac{2}{N}   } w_+^2 d \rho    dt .
 \end{equation}
  We proceed with $I_{2b}$ in the following way 
  $$
\begin{array}{lll}      \displaystyle 
I_{2b}  & = & \displaystyle \left(\delta-(2-p)\frac{N-1}{N} - \frac{2}{N}    \right)\frac{\chi_{_N}}{p}
    \int_{I_T}  \rho^{ (2 -p)\frac{N-1}{N} +\frac{2}{N} -\delta -1  }   (\phi^{p}  - U^{p} )_+ w_+
    d\rho     dt   \\  [4mm]
    &\leq  &  \displaystyle  c     \int_{I_T}  \rho^{ (2 -p)\frac{N-1}{N} +\frac{2}{N} -\delta -1  }   \phi^{p-1}  w^2_+
    d\rho    dt 
    \\  [4mm]
    &\leq  &  \displaystyle c    \int_{I_T}  a^{1-p}(t)  \rho^{ (2 -p)\frac{N-1}{N} +\frac{2}{N} -\delta -1   +\gamma(p-1) }    w_+^2
    d\rho    dt  \end{array} $$
    since $$(2 -p)\frac{N-1}{N} +\frac{2}{N} -\delta -1   +\gamma(p-1) 
    =(p-1)( \gamma - \frac{N-1}{N} ) +\frac{1}{N} -\delta
    \geq - \delta + \frac{1}{N}
    $$
it results 
 $$I_{2b} \leq      \int_{I_T}  c a^{1-p} (t) \rho^{ \frac{1}{N} -\delta    }   w_+^2
    d\rho     dt  .  $$
  Then, thanks to Lemma \ref{lemma4.3},  
    we have that 
   $$I_{2b} \leq     \int_{I_T} \rho_0^{\frac{1}{N}}ca^{1-p}(t)    \rho^{   -\delta    }   w_+^2
    d\rho     dt +     \int_{I_T}   ca^{1-p} (t) \rho_0^{\frac{1}{N}-2}\rho^{2-\delta    }   w_+^2
    d\rho     dt 
    $$
     where   $$ \rho_0^{\frac{1}{N} } \leq  \frac{(2-\delta)(\delta-1)a^{(p-1)}(t) }{8 c}, 
   \quad \mbox{ for } t\leq T.
    $$  
   Then,  
 \begin{equation} \label{I2} \begin{array}{lll}   \displaystyle 
 I_2   \leq \frac{1}{2}     \int_{I_T} \rho^{2 -\delta   }  (w_+)_{\rho}^2 d\rho     dt 
& + & \displaystyle        \int_{0}^T c^{\prime} a^{2N( 1-p)  }  (t) \int_{I} \rho^{\frac{2}{N} - \delta }  w_+^2d\rho    dt  \\ [4mm]  & +  & \displaystyle   
 \frac{(2-\delta)(\delta-1)}{8 }    \int_{I_T}  \rho^{   -\delta    }   w_+^2
    d\rho     dt  .  \end{array}
 \end{equation}
We consider now 
 $$I_3=M
  \chi_{_N}      \int_{I_T}   \rho^{(2-p) \frac{N-1}{N}  + \frac{2}{N}- \delta }   \left(   \left|  {\phi} \right|^{p-2}  \phi-   \left|  {U} \right|^{p-2}  U\right) w_+ d\rho    dt .
  $$
  Thanks to Lemma \ref{lemma4.1},  
     $$
    \int_{I_T}  \rho^{(2-p) \frac{N-1}{N}+ \frac{2}{N} -\delta  }  
     ( |\phi |^{p-1}  - |U|^{p-1} )_+ w_+  d\rho    dt 
     \leq   c     \int_{I_T} \rho^{(2-p) \frac{N-1}{N}+ \frac{2}{N} -\delta  }  
     |\phi |^{p-2}    w_+^2  d\rho    dt 
  $$
  and 
  $$ \phi^{p-2} \leq c \rho^{ 
\gamma(p-2)} (1-\rho)^{p-2}  $$
we have 
$$I_{3} \leq c  \int_I  \rho^{(2-p) \frac{N-1}{N}+ \frac{2}{N} -\delta  +
\gamma(p-2)} (1-\rho)^{p-2}  w_+^2  d\rho    dt .
  $$
  Since $$\gamma \leq \frac{N+1}{N}, \quad \mbox{ and } \    p<2.  $$ we have that 
\begin{equation} \label{789}  (2-p) (\frac{N-1}{N}- \gamma) +  \frac{2}{N} -\delta 
\geq \frac{p-1}{N} - \delta,  \end{equation} 
    and therefore
    $$I_{3} \leq c      \int_{I_T}  \rho^{  \frac{p-1}{N}  -\delta  } (1-\rho)^{p-2}  w_+^2  d\rho    dt .
  $$ 
    Notice that the previous integral presents two possible singularities, at $\rho=0$ and at $\rho=1$. To treat them,   we first consider the positive number  $\rho_0^{\prime}<\frac{1}{2}$ small enough such that 
    $$  c_1 (1- \rho_0^{\prime})^{ \delta-2} \frac{( \rho_0^{\prime})^{p}}{p-1} <\frac{1}{4}$$
  and   we  split  
  the previous integral into there  parts in the following way 
  $$
  \begin{array}{lll} 
  I_{3} &  \leq  & \displaystyle c     \int_{0}^{T} \int_{0}^{\frac{1}{2} }  \rho^{  \frac{p-1}{N}  -\delta  } (1-\rho)^{p-2}  w_+^2  d\rho    dt  \\
[4mm]  & & \displaystyle
 +  c   \int_0^{T} \int_{\frac{1}{2}}^{1- \rho_0^{\prime}} 
  \rho^{  \frac{p-1}{N}  -\delta  } (1-\rho)^{p-2}  w_+^2  d\rho     dt  \\ [4mm]  & & \displaystyle
+ c   \int_0^{T} \int_{1- \rho_0^{\prime}}^{1} 
   \rho^{  \frac{p-1}{N}  -\delta  } (1-\rho)^{p-2}  w_+^2  d\rho      dt .
\end{array}
  $$
  Now we denote the previous integrals  by $I_{3a}$, $I_{3b}$  and $I_{3c}$    respectively,  i.e. 
 $$\begin{array}{lll}
  I_{3a} & := & \displaystyle c \int_0^T   \int_{0}^{\frac{1}{2}}   \rho^{  \frac{p-1}{N}  -\delta  } (1-\rho)^{p-2}  w_+^2  d\rho    dt 
   ,  \\ [4mm]
  I_{3b} & :=  & \displaystyle   c  \int_0^T  \int_{\frac{1}{2}}^{1- \rho_0^{\prime}} 
   \rho^{  \frac{p-1}{N}  -\delta  } (1-\rho)^{p-2}  w_+^2  d\rho    dt 
 ,  \\ [4mm]
  I_{3c} & :=  & \displaystyle   c    \int_{0}^T  \int_{1- \rho_0^{\prime}}^{1} 
   \rho^{  \frac{p-1}{N}  -\delta  } (1-\rho)^{p-2}  w_+^2  d\rho    dt 
  \end{array}
      $$
to write 
 $$I_3= I_{3a} + I_{3b} +I_{3c} .$$
 We consider  first $I_{3a}$,  which satisfies 
$$ \begin{array}{lll} I_{3a}    & \leq  & \displaystyle   c   \int_0^{T}  \int_{0}^{\frac{1}{2}}  \rho^{   \frac{p-1}{N} -\delta 
}   w_+^2  d\rho      dt 
\\ [ 4mm] & \leq  &  \displaystyle c     \int_{I_T}  \rho^{   \frac{p-1}{N} -\delta 
}   w_+^2  d\rho      dt 
.
\end{array}
$$ We now apply  Lemma \ref{lemma4.3},  for $\epsilon_1= \frac{p-1}{N}$ and   $\kappa= \frac{2}{N}$,  and $\rho$ such that $\rho_0^{\frac{p-1}{N}}< \frac{(2- \delta)(\delta-1) }{8} $    to obtain  
\begin{equation} \label{678}
I_{3a}    \leq      \frac{(2-\delta )(\delta -1)}{8}
    \int_{I_T}  \rho^{   -\delta 
}   w_+^2  d\rho     dt  +  c 
    \int_{I_T}   \rho^{  \frac{2}{N} -\delta 
}   w_+^2  d\rho    dt .
\end{equation} 
In  similar way we have that 
\begin{equation} \label{nnn} I_{3b} \leq c   \int_0^{T} \int_{\frac{1}{2}}^{1- \rho_0^{\prime}}  \rho^{ \frac{p-1}{N}- \delta }    w^2_+ d\rho    dt 
 \leq    2^{\frac{3-p}{N} } c     \int_{I_T}   \rho^{ \frac{2}{N}- \delta }    w^2_+ d\rho    dt 
 .
 \end{equation}
 Now,  we consider  $I_{3c}$,  
 then
  $$
  \begin{array}{lll}
  I_{3c}  &  \leq  & \displaystyle  c  \int_0^T  \int_{1- \rho_0^{\prime}}^1  \rho^{ \frac{p-1}{N}- \delta }  (1- \rho)^{p-2}  w^2_+ d\rho     dt  \\ [4mm] 
 &  \leq  & \displaystyle    (1- \rho_0^{\prime})^{\frac{p-1}{N}- \delta}  c  \int_0^T \int_{1- \rho_0^{\prime}}^1   (1- \rho)^{p-2}  w^2_+ d\rho     dt  \\ [4mm]
  &  \leq  & \displaystyle   c_1   \int_0^T \int_{1- \rho_0^{\prime}}^1  (1- \rho)^{p-2}  \| w_+\|_{L^{\infty}(1- \rho_0^{\prime}  ,1)}^2 d \rho      dt .    \end{array}
$$ After integration in the term  $$  \int_{1- \rho_0^{\prime}}^1  (1- \rho)^{p-2}d \rho$$ we have 
$$I_{3c} \leq  
c_1    \frac{ (\rho_0^{\prime})^{p-1}}{p-1}   \int_0^{T}   \| w_+\|_{L^{\infty}(1- \rho_0^{\prime}  ,1)}^2    dt .
$$
Thanks to  the embedding  $L^{\infty}(I) \subset H^1(I)$ for any bounded one-dimensional interval $I$,  we have, in view of $w=0$ at $\rho=1$,  that   $$ \| w_+\|_{L^{\infty}(1- \rho_0^{\prime}  ,1)}^2
\leq  \rho_0^{\prime} \int_{1- \rho_0^{\prime}}^1     | (w_+)_{\rho} |^2 d\rho.$$
Then
 $$
  \begin{array}{lll}
  I_{3c}  
 &  \leq  & \displaystyle   c_1  \frac{ (\rho_0^{\prime})^{p}}{p-1}     \int_{0}^T  \int_{1- \rho_0^{\prime}}^1  | (w_+)_{\rho} |^2d\rho    dt  \\  [4mm] 
 &  \leq  & \displaystyle  c_1 (1- \rho_0^{\prime} )^{ \delta-2} \frac{ (\rho_0^{\prime})^{p}}{p-1}  
     \int_{I_T}  \rho^{2- \delta}     | (w_+)_{\rho} |^2 d\rho     dt 
 \end{array}
$$ 
 i.e. $$
 I_{3c} \leq c_1 (1- \rho_0^{\prime})^{ \delta-2} \frac{( \rho_0^{\prime})^{p}}{p-1}      \int_{I_T} \rho^{2- \delta}    | (w_+)_{\rho} |^2 d\rho dt $$
we take $\rho_o^{\prime}$ small enough such that 
$$  c_1 (1- \rho_0^{\prime})^{ \delta-2} \frac{( \rho_0^{\prime})^{p}}{p-1} <\frac{1}{4}$$
which implies, \begin{equation} \label{678c} 
 I_{3c} \leq \frac{1}{4}      \int_{I_T}  \rho^{2- \delta}    | (w_+)_{\rho} |^2 d\rho dt.  \end{equation} 
 thanks to (\ref{678}), (\ref{nnn}) and (\ref{678c})   
\begin{equation} \label{I3}
I_3 \leq   \frac{(2- \delta)(\delta-1)}{8}       \int_{I_T}  \rho^{   -\delta 
}   w_+^2  d\rho    dt  + c     \int_{I_T} \rho^{ \frac{2}{N}- \delta }    w^2_+ d\rho    dt  + \frac{1}{4}       \int_{I_T}   \rho^{2- \delta}    | (w_+)_{\rho} |^2 d\rho     dt .
 \end{equation} 
   \newline 
We replace  (\ref{I1}), (\ref{I2}) and (\ref{I3}) into (\ref{eq1}) to obtain 
  $$ \frac{1}{2} \int_{I} \rho^{\frac{2}{N} - \delta }  w_+^2d\rho \Big|_T  + \frac{1}{4}
    \int_{I_T} \rho^{2 -\delta   }   | (w_+)_{\rho} |^2 d\rho     dt 
  \leq      \int_{I_T}   c( a(t))   \rho^{\frac{2}{N} - \delta }  w_+^2d\rho    dt
  .
  $$
  Notice that,  the terms  $$  
 \frac{(\delta-2) (\delta -1)}{8} 
  \int_{I_T}\rho^{ - \delta }  w_+^2d\rho  dt
  $$
  in $I_2$ and $I_3$ are cancelled with the term in $I_1$ 
  $$ -\frac{(2-\delta)(\delta -1)}{2}     \int_{I_T}  \rho^{ - \delta }  w_+^2d\rho 
  dt 
  $$
thanks to (\ref{delta}). 
 Now,    Gronwall$^{\prime}$s lemma provides us 
  $$     \int_{I_T}  \rho^{\frac{2}{N} - \delta }  w_+^2d\rho  dt =0$$ which contradicts (\ref{contradiction}) and the proof ends   
   for $t <T_{bu}.$  \qed
       \end{proof}

  \begin{lemma} \label{lemma4.5}  Let $p$ and $M$ be positive numbers satisfying   (\ref{p})  and (\ref{M}) respectively.  Then,  the problem 
   (\ref{eqU}) has  at  most one solution under assumptions 
  \begin{equation} \label{datoini2} 
  U(0, \rho) \geq  \phi(0, \rho),  
  \end{equation}
\begin{equation} U  \leq   A  \rho^{  \frac{N-1}{N}}, 
\label{B} \end{equation}  for a positive constant $A$ large enough. 
  \end{lemma}
  \begin{proof}
  We proceed by contradiction, and assume that there exists two different solutions  of (\ref{eqU}) with the same initial data 
  $U_1$ and $U_2$. We define the function 
  $$w=U_1- U_2$$
  which satisfies,    in view of positivity of $U_1$ and $U_2$
 \begin{equation} \label{mas} 
 w_t  -  \rho^{\frac{2N-2}{N} } {w}_{\rho \rho }= 
 \frac{ \chi_{_N} }{p}   \rho^{(2-p) \frac{N-1}{N}   }   
  \left(  (U_1^{p})_{\rho} - (U_2^{p})_{\rho}   \right) 
  + \frac{M
  \chi_{_N}}{p}   \rho^{(2-p) \frac{N-1}{N}  }    \left(   U_1^{p-1}   -   U_2^{p-1}   \right).
  \end{equation}
 Notice that, since $ p \geq \frac{N}{N-1}$  we have that    $$(2-p) \frac{N-1}{N}+ \frac{2}{N} < 1.$$
 We
   multiply by $\rho^{\frac{2}{N}- \delta} w_+$ for some $\delta>0$ satisfying 
  \begin{equation} \label{delta2} \delta  \in (1,\gamma) 
\end{equation}
for $\gamma< \gamma_0$ and $\gamma_0$ defined in  (\ref{gamma0}). Then, 
we have,  after integration  over $I_T$ 
 \begin{equation} \label{eq1w}  \begin{array}{ll} \displaystyle 
  \left. \frac{1}{2} \int_I w_+^2 \rho^{\frac{2}{N} - \delta } d\rho \right|_T \   \leq   \ &  \displaystyle  \int_{I_T}  \rho^{2- \delta}    {w}_{\rho \rho } w_+ d\rho dt
 \\ [4mm] & \displaystyle 
 + \frac{ \chi_{_N} }{p}  \int_{I_T} \rho^{(2-p) \frac{N-1}{N}  +\frac{2}{N} -\delta }  
  \left( (  U_1^{p})_{\rho} - (U_2^{p})_{\rho})   \right) w_+ d\rho dt \\ [4mm] & \displaystyle 
  +  M
  \chi_{_N} \int_{I_T}  \rho^{(2-p) \frac{N-1}{N}  + \frac{2}{N}- \delta }   \left(   U_1 ^{p-1}  -    U_2  ^{p-1}   \right) w_+ d\rho dt.
\end{array}
  \end{equation} 
 We label  the integrals in the previous equation using the same notation   as in the previous lemma, 
  $$\begin{array}{lll}
  I_1 & := & \displaystyle \int_{I_T} \rho^{ 2- \delta }  {w}_{\rho \rho } w_+ d\rho dt,  \\ [4mm]
  I_2 & :=  & \displaystyle \frac{ \chi_{_N} }{p} \int_{I_T}  \rho^{(2-p) \frac{N-1}{N}  +\frac{2}{N} -\delta }  
  \left( (  U_1^{p})_{\rho} - (U_2^{p})_{\rho})   \right) w_+ d\rho dt,
  \\  [4mm] I_3 &  := & \displaystyle 
   M
  \chi_{_N} \int_{I_T}  \rho^{(2-p) \frac{N-1}{N}  + \frac{2}{N}- \delta }   \left(   U_1 ^{p-1}  -    U_2  ^{p-1}   \right) w_+ d\rho dt.
  \end{array}
      $$
We proceed as in Lemma \ref{lemma4.4} to obtain, thanks to
$$\lim_{\rho \rightarrow 0^+}   \rho^{2- \delta} w_+ (w_+)_{\rho} 
= 
\lim_{\rho \rightarrow 1^-}   \rho^{2- \delta} w_+ (w_+)_{\rho} =0.
$$ which is obtained  as a consequence of 
$$|U|\leq A\rho^{\frac{N-1}{N}}, \quad  \delta <1+ \frac{1}{N} $$   therefore the terms in the boundary are null in the following integrals 
\begin{equation} \label{mas1} 
I_1 = -\int_{I_T}  \rho^{2- \delta }  |{w_+}_{\rho }|^2d\rho +   \frac{(2- \delta)(1-\delta ) }{2} \int_I  \rho^{-\delta }  (w_+)^2  d\rho dt.
   \end{equation}

  The term 
  $I_{2}$ is treated  in the following way.
  First,  we apply Mean Value Theorem to the term $U_1^p-U_2^p$ to obtain, thanks to assumption (\ref{B}), that 
   $$
     (U_1^p-U_2^p )_+\leq (p-1)A^{p-1} \rho^{\frac{(p-1)(N-1)}{N}}(U_1- U_2)_+, $$ then  
 $$\begin{array}{lll}I_{2a} & := & \displaystyle \frac{\chi_{_N}}{p} \int_{I_T}  \rho^{(2-p) \frac{N-1}{N}+ \frac{2}{N} -\delta }  
   ( U_1^{p}  - U_2^{p} )_+   (w_+)_{\rho}  d\rho   dt\\ [4mm] 
    & \leq &  \displaystyle c  \int_{I_T}  \rho^{(2-p) \frac{N-1}{N}+ \frac{2}{N} -\delta  }  
    | A \rho^{\frac{N-1}{N}} |^{p-1}  w_+  |(w_+)_{\rho}|   d\rho  dt 
  \end{array}
  $$
    which implies  
      $$\begin{array}{lll}I_{2a} & \leq & \displaystyle 
 c_p A^{p-1}  \int_{I_T}  \rho^{(2-p) \frac{N-1}{N}+ \frac{2}{N} -\delta  + (p-1) \frac{N-1}{N} }  
   w_+  |(w_+)_{\rho}|   d\rho dt
 \\ [4mm] & \leq & \displaystyle 
 c_p A^{p-1}  \int_{I_T}  \rho^{1+ \frac{1}{N} -\delta  }  
   w_+  |(w_+)_{\rho}|   d\rho dt
 \\ [4mm] 
    & \leq & \displaystyle  \epsilon_0 \int_{I_T}  \rho^{2  -\delta }  
   |(w_+)_{\rho}|^2   d\rho dt
    +  \frac{c_p A^{2p-2} }{4 \epsilon_0}  \int_{I_T}  \rho^{-\delta  +\frac{2}{N}  }  
    w_+^2    d\rho dt.
    \end{array}
   $$ 
i.e.
   $$ I_{2a} \leq  \epsilon_0 \int_{I_T}  \rho^{2  -\delta }  
   |(w_+)_{\rho}|^2   d\rho dt
    +  \frac{c_p A^{2p-2} }{4 \epsilon_0}  \int_{I_T}  \rho^{-\delta  +\frac{2}{N}  }  
    w_+^2    d\rho dt. $$

$I_{2b}$  has the following expression 
  $$I_{2b}=       -((2-p) \frac{N-1}{N}+ \frac{2}{N} -\delta )\frac{\chi_{_N}}{p}
 \int_{I_T}  \rho^{(2-p) \frac{N-1}{N}+ \frac{2}{N} -\delta  }  
   (U_1^{p}  - U_2^{p} )_+   w_+  d\rho dt.
      $$
      Thanks to Mean Value theorem we
      get 
      $$I_{2b}=   c  \int_{I_T}  \rho^{(2-p) \frac{N-1}{N}+ \frac{2}{N} -\delta  }  |\xi^{p-1}|
     w_+^2  d\rho  dt
      $$
      for some $\xi < A\rho^{ \frac{N-1}{N}} $ with $A$ large enough (see assumption (\ref{B})). 
Then, 
$$ \begin{array}{lll}
I_{2b} & \leq & \displaystyle  c A^{p-1}   \int_{I_T} \rho^{ (2-p) \frac{N-1}{N}+ \frac{2}{N} -\delta  +(p-1) \frac{N-1}{N} }  
   w_+^2   d\rho  dt
  \\ [4mm]  & =& \displaystyle
c  A^{p-1}  \int_{I_T} \rho^{1+ \frac{1}{N}  -\delta    }  
   w_+^2  d\rho  dt
    \\   [4mm]  & \leq & \displaystyle
 c   A^{p-1}   \int_{I_T}  \rho^{\frac{2}{N}  -\delta    }  
   w_+^2  d\rho dt
  \end{array}
   $$ 
     and we get 
  \begin{equation} \label{mas2}  
  I_2 \leq     \frac{1}{2} \int_{I_T}  \rho^{2  -\delta }  
   |(w_+)_{\rho}|^2   d\rho
    +  \left( \frac{c_p A^{2p-2} }{2} + c A^{p-1} \right)   \int_{I_T}  \rho^{-\delta  +\frac{2}{N}  }  
    w_+^2    d\rho.
  \end{equation}
  Thanks to Mean Value Theorem,   $I_3$  satisfies  
 $$\begin{array}{lll}
  I_3 & = & \displaystyle 
   M
  \chi_{_N} \int_{I_T}  \rho^{(2-p) \frac{N-1}{N}  + \frac{2}{N}- \delta }   \left(   U_1 ^{p-1}  -    U_2  ^{p-1}   \right) w_+ d\rho dt
  \\ [4mm] 
    & \leq & \displaystyle 
      (p-1) M
  \chi_{_N} \int_{I_T}  \rho^{(2-p) \frac{N-1}{N}  + \frac{2}{N}- \delta }   \left| \phi      \right|^{p-2}  w_+^2  d\rho dt
   \\ [4mm] 
    & \leq & \displaystyle 
      (p-1) M
  \chi_{_N} \int_{I_T}  \rho^{(2-p) \frac{N-1}{N}  + \frac{2}{N}- \delta + (p-2) \gamma  } (1- \rho)^{p-2}    w_+^2  d\rho dt.
    \end{array}
    $$
    We now proceed as in Lemma \ref{lemma4.4} and it results  
\begin{equation} \label{mas3}  I_3 \leq c \rho_0^{\frac{p}{N}}   \int_{I_T} \rho^{- \delta }  w_+^2d\rho dt+  c    \int_{I_T} \rho^{ \frac{2}{N}- \delta }  w_+^2d\rho dt , 
\end{equation} 
for $\rho_0$ such that $\rho_0^{\frac{p}{N}} \leq \frac{(2- \delta)(\delta -1) }{8 c }$.  
   Then, we replace  (\ref{mas1}), (\ref{mas2}) and (\ref{mas3}) into  (\ref{mas}) to obtain 
  $$ \left.  \int_I \rho^{ \frac{2}{N}- \delta }  w_+^2d\rho\right|_{T}  \leq c   \int_{I_T} \rho^{ \frac{2}{N}- \delta }  w_+^2d\rho dt , $$
 and   Gronwall Lemma ends the proof. 
     \qed
  \end{proof} 
    
    \begin{lemma} \label{lemma4.62}  Let $p$ and $M$ be positive numbers satisfying   (\ref{p})  and (\ref{M}) respectively.  Then,  
     the solution $U$ to (\ref{eqU})-(\ref{iniU}),  is a classical solution in $(0,T)$ (for $T<T_{bu}$)  in the sense 
 \begin{equation} \label{regul2} U \in C^{1, 2}_{loc}(I_{T}) \cap C^{\alpha}(I_{T}) ,  \quad \mbox{ for } \alpha = 1-  \frac{N+3}{q} \end{equation}
 and $q \frac{N+2}{2-p} $,   where  $U_0$ is defined in  (\ref{iniU}) for  $u_0$ satisfying  (\ref{Hdato1}), (\ref{id23}) and (\ref{Hdato2}). 
   \end{lemma}
  \begin{proof}
  As in Section \ref{s2.1}  we introduce the function $W$,  
   $$ W(t,s)= s^{-N}U(t, s^N) $$
   which satisfies 
   (\ref{eqU22}),   and $\tilde{W}: (0,T) \times B_{N+2} \rightarrow \R$,  defined by $\tilde{W}(t,x)=W(t, |x|)$ which satisfies 
   $$\tilde{W}_t - N^{-2} \Delta_{N+2} \tilde{W}+ b(t,x) \nabla \tilde{W} = f $$
   where 
   $$ b:=  \chi _{_N}    |x|^{p-2   }  \frac{x}{N}  \left| \tilde{W} \right|^{p-2}
 \tilde{W}, 
   \qquad 
   f:=
  \chi _{_N}    |x|^{p-2   }   \left( \tilde{W}+M  \right)   \left|  \tilde{W} \right|^{p-2}
 \tilde{W} .
 $$
  We notice that $b$ is bounded and continuous  and in view of $U \leq c \rho$, $f$ satisfies 
  $$|f|  \leq c |x|^{p-2} 
  $$
  therefore $$f\in L^{\infty}(0,T: L^q(B_{N+2})  )
  $$ for any $q <\frac{N+2}{2-p} $. Therefore we have that 
  $$\tilde{W} \in W^{1,q}(0,T: L^q(B_{N+2})) \cap   L^q(0,T:W^{2,q}(B_{N+2}))$$
  see for instance Quittner-Souplet \cite{qs},  Theorem  48.1 p. 438. 
  Since $$ \frac{N+2}{2-p} >\frac{N+2}{2-\frac{N}{N-1}} = \frac{(N+2)(N-1)}{N-2} >N+3,$$
and  thanks to the Sobolev embedding  (see for instance Theorem 5.4, Adams \cite{adams}, p. 97) we have  that 
  $$\tilde{W} \in C^{0,\alpha}((0,T) \times B_{N+2})  )$$
  for $\alpha \leq 1- \frac{N+3}{q}.$ It implies that 
   \begin{equation} \label{regularidad}  U \in C^{0, \alpha}(I_{T}) \end{equation}  for  any $T <T_{bu}$. Then, for any $\epsilon >0$ we have that 
   $$g(t):=U(t,\epsilon) \in C^{0, \alpha}(0,T_{bu})$$ and $U$ satisfies 
   $$ \left\{ \begin{array}{l}  \displaystyle  U_{t} - \rho^{\frac{2N-2}{N}} U_{\rho \rho} + \tilde{b}   {U}_{\rho} = \tilde{f}, \qquad   \epsilon<\rho<1, \quad 0<t<T_{bu} 
  ,  \\ [4mm]  \displaystyle 
  U(t,\epsilon)= g(t), \quad U(1,t)=0, 
 \\ [4mm]  \displaystyle  U(0, \rho)= U_{0}(\rho), \quad    \epsilon<\rho<1. 
 \end{array} \right. $$
  for 
$$\tilde{b}:= \chi _{_N}    \rho^{(2-p) \frac{N-1}{N}}  \left|  U \right|^{p-2}U$$
$$\tilde{f}: = \chi _{_N}    \rho^{(2-p) \frac{N-1}{N}   }  M  \left|  U \right|^{p-2}U.$$
Since $\tilde{b}, \tilde{f}  \in C^{0,\beta} ( (0,T) \times (\epsilon,1)) $ for $\beta = (p-1) \alpha$,  we have, thanks 
  to Theorem 5.6, Lieberman \cite{lieberman}, p. 90, that
  $$ U \in C_{t,x}^{1,2}( (0,T) \times (\epsilon,1)), $$ which implies, in view of (\ref{regularidad}),   the wished regularity.
   \qed
\end{proof}

  %
  %

  %
  %
  %
  %
  %
  %
  %
  %
  %
  %
  %
  %
  %
  %
  %
  %
  %
  %
  %
  
  %
  %
  %
  
  %

  %
  %
  %
  %
  %
  %
  %
  %
  %
  %
  %
  %
  %
  %
  %
  %
  %
  %
  %
  
  %
  %
  %
  
  %
 \section{Blow up of solutions. Proof of Theorem \ref{t1}.} \label{s6}  
 \setcounter{equation}{0}
 
  We  consider the function
$$ \phi(\rho ,t)=  \left\{
 \begin{array}{l}   \displaystyle 
 \phi_1(t, \rho)  ,  \quad  0<\rho \leq \rho_1,  
\\ [4mm] \displaystyle 
 \phi_2(t,\rho), 
 \quad   \rho_1 <\rho < 1, 
\end{array}
 \right.
$$
for $\phi_1$ and $\phi_2$ defined in (\ref{phi1}) and (\ref{phi2}) respectively. 
Notice that, thanks to Lemma \ref{l3.1} and Lemma \ref{l3.2} we have that 
$$\mathcal{L}(\phi) \leq 0, \qquad \mbox{ for }  \rho \neq \rho_i, \quad  (i=1,2).$$
Since 
$$\lim_{\rho\rightarrow \rho_1^{-}} \phi_{1\rho } (t, \rho ) =   \frac{\gamma a \rho_1^{\gamma-1}}{(\rho_1^{\gamma} + a)^2}  \leq   \phi_{1 } (t, \rho_1) \frac{\gamma  }{\rho_1 }   $$
and   
$$\phi_{2\rho } (t, \rho)
 =  \left\{ \begin{array}{ll}   - \beta( 1 + \kappa \frac{2 \rho -(\rho_2 + \rho_1) }{(\rho_2- \rho_1)} )  ,   &  \rho_1 <\rho< \rho_2
 \\  - \beta  ,  &  \rho_2 <\rho< 1
   \end{array} \right. $$
we have that 
$$ \lim_{\rho\rightarrow \rho_1^{+}} \phi_{2\rho } (t, \rho)
 =     ( \kappa-1)\beta(t)=  \frac{( \kappa-1)}{1- \rho_1} \phi_1(t, \rho_1).$$
In view of 
$$ \frac{ \kappa-1}{1- \rho_1} = \frac{\gamma  }{\rho_1 } $$
i.e. 
$$ \kappa =  1+  \gamma \frac{(1- \rho_1)}{\rho_1}= 1+ \gamma$$ we have that $\phi_{\rho}$ is a continuous function at $\rho= \rho_1$. 
In the same way we compute  $\phi_{2\rho } $ at  $ \rho_2$ to get 
$$\lim_{\rho\rightarrow \rho_{2}^{-}} \phi_{2\rho } (t, \rho ) =  - \beta( 1 + \kappa ) < - \beta = \lim_{\rho\rightarrow \rho_{2}^{+}} \phi_{2\rho } (t, \rho )$$
which implies  that the first derivative of $\phi_2$ respect to $\rho$ presents a positive jump at $\rho= \rho_2$ and therefore
$$\mathcal{L}(\phi) \leq 0, \qquad \mbox{ for }  \rho  \in   (0,1).$$
Lemma \ref{lemma4.4} and Lemma \ref{lemma4.5}  provide that for any regular  solution $U$ to (\ref{eqU}) with initial data 
$$U(0, \rho) \geq \phi(0,\rho)$$
  satisfies 
\begin{equation}
 \label{5.1} 
U (t, \rho) \geq \phi(t,\rho), \quad \mbox{ for }  t<T_{max}. \end{equation}
Notice that for $\rho= |a(t)|^{\frac{1}{\gamma}}$ we have 
$$\lim_{ t \rightarrow T_{max} }   \phi(t,  |a(t)|^{\frac{1}{\gamma}})=\frac{1}{2}$$
and $$   \lim_{ t \rightarrow T_{max} } a(t)=0.$$
Following standard arguments (see for instance \cite{jl}) and thanks to Lemma \ref{lemma4.4} we end the proof.  
\qed
\textsc{ \underline{Acknowledgment}.}

The author wants to  thank to   the anonymous  reviewers and also to  professor  Michael Winkler, for their   helpful comments and suggestions.
The author is  supported by Ministerio  de Ciencia e Innovaci\'on, Spain, under grant mumber
  MTM2017-83391-P.

\end{document}